\UseRawInputEncoding
\documentclass[12pt]{article}
\usepackage{amsmath,amssymb,amsthm,graphicx,hyperref,array,booktabs,microtype,longtable}
\usepackage[margin=0.9in]{geometry}
\newtheorem{theorem}{Theorem}[section]

\newtheorem{remark}[theorem]{Remark}
\newtheorem{lemma}[theorem]{Lemma}
\newtheorem{corollary}[theorem]{Corollary}
\newtheorem{proposition}[theorem]{Proposition}

\newtheorem{conjecture}[theorem]{Conjecture}
\DeclareRobustCommand{\sheet}[1]{\texttt{ancillary/#1.csv}}

\hypersetup{
    colorlinks=true,
    linkcolor=black,
    citecolor=black,
    urlcolor=blue,
}

\title{Exact Second-Order Zarankiewicz Numbers for Complete-Graph Incidence Families}
\author{Yannan Chen\footnote{School of Mathematical Sciences, South China Normal University, Guangzhou 510631, China ({\tt ynchen@scnu.edu.cn}).}
\and
Johan L\"ofberg\footnote{Division of Automatic Control, Link\"oping University, SE-581 83 Link\"oping, Sweden ({\tt johan.lofberg@liu.se}).}
\and
Liqun Qi\footnote{Jiangsu Provincial Scientific Research Center of Applied Mathematics, Nanjing 211189, China.
Department of Applied Mathematics, The Hong Kong Polytechnic University, Hung Hom, Kowloon, Hong Kong.
({\tt maqilq@polyu.edu.hk})}
}
\date{\today}

\begin{document}

\maketitle

\begin{abstract}
Let \(n\ge6\), \(m=\binom n2=\frac{n(n-1)}2\), and consider the complete-graph incidence family on \(K_n\): the columns are the vertices of \(K_n\), the rows its edges, and the one-edge graph is the incidence graph. The universal cell bound of L\"ofberg and Qi reads
\[
z_2(m,n)\le Z(n):=\left\lfloor\frac{n(n-1)(n+2)}4\right\rfloor
=\begin{cases}
\dfrac{n(n-1)(n+2)}4, & n\not\equiv3\pmod4,\\[1mm]
\dfrac{n(n-1)(n+2)-2}4, & n\equiv3\pmod4 .
\end{cases}
\]
We implement the nested one-factorization construction of that family and determine exactly what it certifies. For \(n=2q\) with \(q\) an odd prime the construction has no hole, every pair of factors \(F_a\cup F_b\) is a genuine Hamilton cycle, the cross-factor transfer equations apply verbatim, and \(z_2=z_{SL}=z_{RL}=q(2q-1)(q+1)\) for \(q\ge5\). For odd \(n=2p+1\) the near-perfect one-factorization yields Hamilton paths \(F_a\cup F_b\), and the virtual edge closing them into an odd cycle breaks the transfer argument: that scheme certifies only \(R(G_p)\le Z(n)\). The orders \(n=7,8,9,12,13,16,17,18,20,21\), at which neither regular branch is available, are settled by explicit configurations of a different shape: each attains the cell bound and satisfies \((\mathrm{RW}3^+)\), and each is obtained from a configuration of a larger order by deleting one or more vertex stars and repairing the restricted grid, the parent being certified except at \(n=7,16,17\): at \(n=7\) and \(n=16\) the grid is inherited from a larger configuration under a fresh pairing of its occupied cells, and the parent of \(n=16\) and \(n=17\) is the nested \(q=9\) witness, which fails \((\mathrm{RW}3^+)\) and is repaired by the operation. At every one of these orders \(z_2=z_{SL}=z_{RL}=Z(n)\), and the exact values are collected in Theorem~\ref{thm:unified}; all remaining orders are recorded as conjectures. The order \(n=7\), the first settled one with \(n\equiv3\pmod4\), is the only one whose grid has a hole, grounded by the zero-companion rule of Section~\ref{sec:prelim}. Machine-readable configurations and the certificate checker accompany the paper.
\end{abstract}



\section{Introduction}

The classical Zarankiewicz problem asks for the maximum number \(z(m,n)\) of \(1\)'s in an \(m\times n\) binary matrix containing no all-one \(2\times2\) submatrix, equivalently the maximum number of edges in a \(C_4\)-free bipartite graph with parts of sizes \(m\) and \(n\). The problem originates from Zarankiewicz \cite{za51}; early fundamental contributions include \cite{kst54,culik56,re58,gu69}. For modern accounts and recent exact values we refer to \cite{ni10,chm24,cc26mx4,qi1}, and for the general extremal theory of \(C_4\)-free graphs to \cite{bo04}.

The connection between the Zarankiewicz problem and the sum-of-squares (SOS) rank of biquadratic forms was developed in \cite{qi1,reproducibility}. L\"ofberg and Qi \cite{reproducibility} introduced the second order Zarankiewicz number \(z_2(m,n)\), the recursive-line Zarankiewicz number \(z_{RL}(m,n)\), and the signed Zarankiewicz number \(z_{SL}(m,n)\). These satisfy the unconditional hierarchy
\begin{equation}\label{eq:hierarchy}
\mathrm{BSR}(m,n)\ge z_2(m,n)\ge z_{SL}(m,n)\ge z_{RL}(m,n)\ge z_{wL}(m,n)\ge z(m,n),
\end{equation}
where \(\mathrm{BSR}(m,n)\) is the biquadratic SOS rank, \(z_{wL}\) is the weak limited augmented Zarankiewicz number, and \(z\) is the classical Zarankiewicz number.

In \cite{reproducibility}, the complete-graph incidence family was used to obtain a cubic asymptotic separation between \(z_2\) and \(z_{wL}\). For \(N=2q\) with \(q\) an odd prime, the columns are the vertices of \(K_{2q}\), the rows are its edges, and the one-edge graph is the incidence graph. Using a nested perfect one-factorization construction, L\"ofberg and Qi proved
\begin{equation}\label{eq:evencase}
z_2\left(\binom{2q}{2},2q\right)=q(2q-1)(q+1)
=\frac{N(N-1)(N+2)}4.
\end{equation}
For \(q\ge5\), the same construction satisfies the ordinary recursive-line criterion \((\mathrm{RW}3^+)\), hence
\[
z_{RL}\left(\binom{2q}{2},2q\right)=z_2\left(\binom{2q}{2},2q\right).
\]
For the exceptional case \(q=3\), ordinary propagation leaves odd-cycle components unresolved; the signed criterion \((RW3^\pm)\) was introduced to certify the \(15\times6\) witness. Recently \cite{ccq26sl} proved that the value at \(q=3\) is nevertheless the same one, \(z_{RL}(15,6)=60\), by exhibiting a different \(15\times6\) witness, a cyclic grid, whose closure does satisfy \((\mathrm{RW}3^+)\). Thus the signed criterion is no longer needed to determine any value in this family, and no separation of \(z_{SL}\) and \(z_{RL}\) arises at any order of it settled so far.

The purpose of the present paper is to treat the full complete-graph incidence family with \(m=\binom n2\) for all \(n\ge6\). The formula under study is
\begin{equation}\label{eq:main}
z_2\left(\binom n2,n\right)=z_{RL}\left(\binom n2,n\right)=\left\lfloor\frac{n(n-1)(n+2)}4\right\rfloor.
\end{equation}
The right-hand side is the universal cell bound of \cite[Proposition 3.2]{reproducibility}, and by Theorem~\ref{thm:unified} it is attained at the orders for which a configuration below is certified by \((\mathrm{RW}3^+)\); the remaining orders are collected in Conjectures~\ref{conj:odd} and \ref{conj:all-n}. The construction is the nested perfect one-factorization construction of \cite{reproducibility} on \(K_n\), with the following parity-dependent behavior:

\begin{itemize}
\item \emph{Even \(n=2q\).} For odd prime \(q\) the perfect one-factorization of \(K_{2q}\) gives genuine Hamilton cycles \(F_a\cup F_b\) for every pair of distinct factors, the construction has no hole, and the cross-factor transfer equations of \cite[\S A.3--A.4]{reproducibility} apply verbatim, so \((\mathrm{RW}3^+)\) holds for \(q\ge5\) (Proposition~\ref{prop:even}); the exceptional prime \(q=3\), that is \(n=6\), is certified only by the signed criterion \cite{reproducibility}, the value \(z_{RL}(15,6)=60\) itself having since been proved in \cite{ccq26sl} with a different witness. For even \(q\), and for odd composite \(q\), the construction needs a perfect one-factorization of \(K_{2q}\) together with a paired certificate; the factorization itself is not conjectural whenever \(2q-1\) is prime, extending each cyclic near-perfect matching on the \(2q-1\) vertices by the edge \(\{a,\infty\}\) at its missing vertex \(a\) giving one, which covers \(K_8\), \(K_{12}\), \(K_{18}\) and \(K_{20}\), and only the remaining orders are subject to the conjecture of Kotzig that every \(K_{2q}\) admits a perfect one-factorization \cite{kotzig64}. That range is reported here as conditional. The even orders of this branch that are nevertheless settled below are \(n=8,12,16,18,20\), by the explicit configurations of Appendix~\ref{app:sheets} (Corollary~\ref{cor:ten}); of these, \(n=18\) and \(n=20\) are obtained from the certified construction at \(n=22\) (\(q=11\)) by deleting stars and repairing the restricted grid, while \(n=16\) starts from the nested construction at \(q=9\), which fails \((\mathrm{RW}3^+)\) and is repaired by the same operation (Appendix~\ref{app:12016}); no hypothesis on the one-factorization of \(K_{2q}\) is used at these three orders.
\item \emph{Odd \(n=2p+1\).} The near-perfect one-factorization of \(K_{2p+1}\) gives paths \(F_a\cup F_b\) with endpoints \(a,b\); adding a virtual edge closes a Hamilton path into an odd cycle, and the cell counting bounds the displayed length by \(R(G_p)\le Z(n)\). This branch is \emph{not} certified. First, \(F_a\cup F_b\) is a Hamilton path precisely when \(\gcd(a-b,n)=1\), so all pairs of factors give a Hamilton path only for \(n\) prime; already at \(n=15\), \(a=0\) and \(b=5\) one gets a path and two cycles (Remark~\ref{rem:path}). Second, the closed walk is odd and does not alternate between the two factors around the virtual edge, so the transfer equations of the even case do not apply verbatim and the wraparound has to be re-derived with explicit boundary cases (Remark~\ref{rem:cross-odd-gaps}(i)). Third, for odd \(p\) exactly one non-incidence cell stays uncovered, and the claim that a chain reaching this hole is grounded is not a reachability proof for all unresolved inner products (Remark~\ref{rem:cross-odd-gaps}(ii)). For even \(p\) the non-incidence cells are even, no hole is available, and the intra-factor row pairing does not exist at all. What remains is the conjecture \(z_2=z_{SL}=z_{RL}=Z(n)\) for \(n=2p+1\) prime with \(p\ge5\) odd (Conjecture~\ref{conj:odd}). Independently of the scheme, the orders \(n=7,9,13,17,21\) are settled by explicit configurations of a different shape (Corollary~\ref{cor:ten} and Section~\ref{sec:examples}): at \(n=7\) by the \(21\times7\) configuration of Appendix~\ref{app:217}, whose grid is the one produced by deleting the star of a vertex of the \(28\times8\) configuration of Appendix~\ref{app:288}, with the occupied cells then paired afresh and the single hole grounded by the zero-companion rule of Section~\ref{sec:prelim}; at \(n=13\) by the \(78\times13\) configuration of Appendix~\ref{app:7813}, obtained from the certified construction at \(n=14\) by deleting the star of a vertex and repairing the orphaned cells; and at \(n=9,17,21\) by the configurations of Appendices~\ref{app:369}, \ref{app:13617} and \ref{app:21021}. The remaining odd orders are part of Conjecture~\ref{conj:all-n}.
\end{itemize}

The grid bookkeeping is parity dependent: \(H=0\) for even \(n\) and for \(n\equiv1\pmod4\), and \(H=1\) for \(n\equiv3\pmod4\). The upper bound is the cell bound of Proposition~\ref{prop:cell}. The matching lower bound is proved in \cite{reproducibility} for \(n=2q\) with \(q\) an odd prime and \(q\ge5\), and for the remaining orders it is established exactly at \(n=7,8,9,12,13,16,17,18,20,21\), the orders of Theorem~\ref{thm:unified}, by the explicit configurations of Appendix~\ref{app:sheets}; elsewhere it is stated as Conjecture~\ref{conj:odd} or Conjecture~\ref{conj:all-n}. At \(n=7\), the first settled order with \(n\equiv3\pmod4\), the configuration has the hole forced by the parity of Lemma~\ref{lem:counts-odd}, and its replay uses the zero-companion rule of Section~\ref{sec:prelim}.

The paper is organized as follows. Section~\ref{sec:prelim} recalls the recursive-line framework, proves the coefficient identities that all later certificates use, and records the cell bound and the classical Zarankiewicz value. Section~\ref{sec:even} treats the even case \(n=2q\), Section~\ref{sec:odd} the odd case \(n=2p+1\), and Section~\ref{sec:unified} combines them into the unified formula \eqref{eq:main}. Section~\ref{sec:examples} records the provenance of the ten configurations that settle the individual orders, and Section~\ref{sec:conclusion} and Section~\ref{sec:open} conclude. Appendix~\ref{app:sheets} fixes the conventions of the replay, collects the shapes of the ten configurations and the output of their replays in Tables~\ref{tab:shape} and~\ref{tab:certs}, and indexes the data sheets; Appendices~\ref{app:369}, \ref{app:288} and \ref{app:217} display the three smallest configurations in full, the seven larger ones being recorded cell by cell in the data sheets cited there. The source of that material is the checker \texttt{ancillary/certificate\_checker.py}: it audits each data sheet, replays the closure of \cite[Definitions 4.1 and 4.5]{reproducibility} and the zero-companion rule of Section~\ref{sec:prelim} to its least fixed point, and tests the terminal conditions, so that every number of Tables~\ref{tab:shape} and~\ref{tab:certs} can be reproduced from the tables alone.

\section{Preliminaries}\label{sec:prelim}

\subsection{Biquadratic forms and SOS rank}

Let \(P(\mathbf x,\mathbf y)\) be an \(m\times n\) biquadratic form. If there exist real bilinear forms \(f_1,\dots,f_r\) such that
\[
P=\sum_{t=1}^r f_t^2,
\]
then \(P\) is a sum of squares, and the minimum such \(r\) is the SOS rank \(\mathrm{SOS}(P)\). The biquadratic SOS rank \(\mathrm{BSR}(m,n)\) is the maximum SOS rank among all \(m\times n\) SOS biquadratic forms; such representations go back to the classical theory of sums of squares of real polynomials \cite{clz95}.

\subsection{Augmented bipartite configurations}

Let \(G=([m],[n],E_1\cup E_2)\) be an augmented bipartite configuration, where \(E_1\) is a set of one-edges (single cells) and \(E_2\) is a set of two-edges (unordered pairs of distinct cells). Let \(u_{ij}=x_iy_j\). The displayed SOS associated with \(G\) is
\[
P_G(\mathbf x,\mathbf y)=\sum_{(i,j)\in E_1}u_{ij}^2
+\sum_{\{(i,j),(k,\ell)\}\in E_2}(u_{ij}+u_{k\ell})^2.
\]
Write \(R(G)=|E_1|+|E_2|\). The configuration satisfies the simplicity condition \((S)\) if the supports of distinct selected edges are disjoint. It is limited if \(|E_1|=z(m,n)\), and irreducible if \(\mathrm{SOS}(P_G)=R(G)\).

\subsection{Coefficient identities and the soundness of the rules}\label{subsec:soundness}

Write \(c=(i,j)\) for a cell and \(u_c=x_iy_j\), so that the displayed SOS of a configuration \(G=([m],[n],E_1\cup E_2)\) reads
\begin{equation}\label{eq:PG}
P_G=\sum_{c\in E_1}u_c^2+\sum_{\{c,d\}\in E_2}(u_c+u_d)^2,\qquad R(G)=|E_1|+|E_2|.
\end{equation}
Let \(\Omega\) be the set of occupied cells, and for distinct cells put \(\delta(c,d)=1\) when \(\{c,d\}\in E_2\) and \(\delta(c,d)=0\) otherwise, the value \(0\) included when one of the two cells is unoccupied. In a representation of \(P_G\) as a sum of \(s\) squares of real bilinear forms, \(f_t=\sum_c v_{c,t}u_c\), write \(v_c\in\mathbb R^s\) for the coefficient vector of the cell \(c\).

\begin{lemma}[Coefficient identities]\label{lem:coeff}
In every such representation,
\begin{equation}\label{eq:coeffnorm}
\|v_c\|^2=\mathbf 1_\Omega(c),
\end{equation}
\begin{equation}\label{eq:coeffline}
\langle v_c,v_d\rangle=\delta(c,d)\qquad\text{if \(c\ne d\) share a row or a column,}
\end{equation}
and for the two diagonals \(\{c,d\}\), \(\{a,b\}\) of a genuine rectangle
\begin{equation}\label{eq:coeffrect}
\langle v_c,v_d\rangle+\langle v_a,v_b\rangle=\delta(c,d)+\delta(a,b).
\end{equation}
\end{lemma}

\begin{proof}
Compare coefficients with \eqref{eq:PG}. A squared cell monomial has coefficient \(\|v_c\|^2\); a product of two distinct cells of a common line has one unordered cell-pair representation; a product of two cells in distinct rows and in distinct columns has exactly the two diagonal representations. Cancelling the factor \(2\) in the last two comparisons gives \eqref{eq:coeffline} and \eqref{eq:coeffrect}. An unoccupied cell contributes no monomial, so \(v_c=0\) there.
\end{proof}

\begin{proposition}[Soundness of a finite certificate]\label{prop:soundness}
Suppose that a finite sequence of the rules of \S\ref{subsec:rw3} identifies both halves of every selected two-edge, keeps distinct selected edges in distinct classes, and certifies every pair of distinct selected-edge classes as orthogonal. Then
\begin{equation}\label{eq:soundness}
\mathrm{SOS}(P_G)=R(G).
\end{equation}
\end{proposition}

\begin{proof}
By \eqref{eq:coeffnorm} the vector of an occupied cell is a unit vector, so inner product \(1\) forces the two vectors to coincide and inner product \(0\) makes them orthogonal: the line rule and the saturation rule are sound. By \eqref{eq:coeffrect} either diagonal of a genuine rectangle determines the other at its own prescribed value, which is the prescribed-value rectangle-transfer rule; when both diagonals are selected two-edges their inner products sum to \(2\) while each is at most \(1\) by Cauchy--Schwarz, so both equal \(1\), and the complementary-pair rule is sound as well. Induction over the finite sequence establishes every identification and every orthogonality that it claims, in every representation of \(P_G\) by \(s\) squares. At termination the \(R(G)\) selected edges furnish \(R(G)\) mutually orthogonal unit vectors in \(\mathbb R^s\), so \(s\ge R(G)\), while \eqref{eq:PG} gives \(s\le R(G)\) for the displayed representation. Hence \eqref{eq:soundness}.
\end{proof}

Holes require no separate theorem. If \(c\notin\Omega\), then \(v_c=0\), and \eqref{eq:coeffrect} for a rectangle of which \(c\) is a corner reduces to
\begin{equation}\label{eq:hole}
\langle v_a,v_b\rangle=\delta(a,b)
\end{equation}
for the diagonal \(\{a,b\}\) opposite to the diagonal through \(c\): that opposite diagonal is certified \emph{at its own prescribed value}, its two halves being identified when it is a selected two-edge and declared orthogonal only when it is not. This corrected conclusion is the fifth rule below; the stronger conclusion \(\langle v_a,v_b\rangle=0\) would be unsound for a rectangle whose opposite diagonal is a selected two-edge. The same soundness proof covers zero, one or several holes: the zero-companion rule is a consequence of \eqref{eq:coeffrect} and not an additional inference principle.

What this replaces. The identities \eqref{eq:coeffnorm}--\eqref{eq:coeffrect} and their soundness argument are stated once, with the recursive-line framework of \cite[Definitions 4.1, 4.5, 4.11]{reproducibility} recalled next. Each finite case then needs only a verified input and the terminal conditions: an accepting finite sequence is sufficient, exhaustive least-fixed-point computation being one way of finding it rather than a further requirement. Printed equal labels specify the intended two-edges; they are never assumed identified before the certificate proves it.

\subsection{The recursive-line criterion \((\mathrm{RW}3^+)\)}\label{subsec:rw3}

On the set \(\Omega\) of occupied cells one constructs the least equivalence relation \(\sim\) and the least symmetric orthogonality relation \(\perp_R\) closed under the following five rules, the first four of which are those of \cite[Definitions 4.1 and 4.5]{reproducibility}.

\begin{enumerate}
\item \emph{Line rule.} If occupied cells \(p,q\) lie in a common row or column, then their prescribed value \(\delta(p,q)\in\{0,1\}\) is certified: add \(p\sim q\) if \(\delta(p,q)=1\) and \(p\perp_R q\) if \(\delta(p,q)=0\). This is \eqref{eq:coeffline}.
\item \emph{Saturation rule.} If \(p\sim p'\), \(q\sim q'\) and \(p'\perp_R q'\), then add \(p\perp_R q\): identified vectors are one and the same vector.
\item \emph{Rectangle transfer rule.} Let \(\{p,q\}\) and \(\{r,s\}\) be the two diagonals of a genuine rectangle, that is of a pair of distinct rows together with a pair of distinct columns. If the companion diagonal \(\{r,s\}\) is certified to have its prescribed value \(\delta(r,s)\), add the same conclusion for the target diagonal \(\{p,q\}\), identification when \(\delta(p,q)=1\) and orthogonality when \(\delta(p,q)=0\). This is \eqref{eq:coeffrect}.
\item \emph{Complementary-pair rule.} If both diagonals \(\{p,q\}\) and \(\{r,s\}\) of a genuine rectangle are selected two-edges, then add \(p\sim q\) and \(r\sim s\) simultaneously; by \eqref{eq:coeffrect} both inner products equal \(1\).
\item \emph{Zero-companion rule for a rectangle with an unoccupied corner.} Let \(\{p,q\}\) and \(\{r,s\}\) be the two diagonals of a genuine rectangle, and suppose that one of the four corners of the rectangle is unoccupied; name it \(p\), so that \(q\) is the corner opposite to it and \(\{r,s\}\) the diagonal opposite to \(\{p,q\}\). Since \(v_p=0\) and \(\delta(p,q)=0\), identity \eqref{eq:coeffrect} reduces to \eqref{eq:hole}: the opposite diagonal \(\{r,s\}\) is certified at its own prescribed value, its two halves being identified when it is a selected two-edge and declared orthogonal only when it is not. An unoccupied corner therefore acts as a grounded zero for the rectangles whose opposite diagonal is unselected, and a rule declaring \(\{r,s\}\) orthogonal unconditionally would be unsound. This is how the single hole of the odd-\(p\) case is used: in the replay of the \(21\times7\) configuration of Appendix~\ref{app:217} the rule is applied to the \(120\) rectangles having the hole as a corner, none of whose opposite diagonals is a selected two-edge, and it is what grounds the \(16\) pairs of selected-edge classes that the other four rules leave uncertified.
\end{enumerate}

A configuration \(G\) satisfies \((\mathrm{RW}3^+)\) if the closure satisfies: every selected two-edge has its two halves identified; distinct selected edges lie in distinct equivalence classes; and every pair of distinct selected edges has orthogonal representatives. If \(G\) satisfies \((\mathrm{RW}3^+)\), then \(P_G\) is irreducible, so \(\mathrm{SOS}(P_G)=R(G)\) \cite[Lemma 2.10]{xu26} or \cite[Proposition 4.6]{reproducibility}; this is the finite form of Proposition~\ref{prop:soundness}.

The recursive-line Zarankiewicz number \(z_{RL}(m,n)\) is the maximum of \(R(G)\) over simple limited configurations whose one-edge graph is \(C_4\)-free and which satisfy \((\mathrm{RW}3^+)\).

\subsection{Certificates as grounded components}\label{subsec:grounded}

Since the audit of Tables~\ref{tab:shape} and~\ref{tab:certs} is to be a shared computation rather than ten separate arguments, we record the normal form in which an accepting replay is stored. First identify, by a legal finite sequence of the preceding rules, the two halves of every selected two-edge, never merging two different selected edges; orthogonality facts may be used along the way. Write \(w_e\) for the common vector of the selected edge \(e\) and introduce, for two distinct selected edges, the unknown
\begin{equation}\label{eq:X}
X_{\{e,f\}}=\langle w_e,w_f\rangle,\qquad \mathcal V=\binom{E_1\cup E_2}{2},
\end{equation}
so that \(|\mathcal V|=\binom{R(G)}2\). For distinct cells put \(D_{cd}=\langle v_c,v_d\rangle-\delta(c,d)\); identity \eqref{eq:coeffrect} reads \(D_{cd}+D_{ab}=0\). After the identifications \(D_{cd}=0\) whenever \(\{c,d\}\) is a selected two-edge or a diagonal through a hole, and \(D_{cd}\) is otherwise one of the unknowns \eqref{eq:X}, so that the residual identities are of the two forms
\begin{equation}\label{eq:residual}
X_A=0\qquad\text{or}\qquad X_A+X_B=0 .
\end{equation}
Make a graph \(\Gamma\) on \(\mathcal V\) whose edges are the transfer relations of the second kind. A vertex is \emph{grounded} if its zero is supplied by a line, by a rectangle one of whose diagonals is a selected two-edge or passes through a hole, or by the identification prefix; each ground and each transfer is to carry its justification, and no root may depend circularly on the propagation it initiates.

\begin{proposition}[Grounded-component criterion]\label{prop:grounded}
If every connected component of \(\Gamma\) contains a grounded vertex, then all distinct selected-edge vectors are orthogonal. Together with the identification prefix this supplies the terminal conditions of \((\mathrm{RW}3^+)\) and hence, by Proposition~\ref{prop:soundness}, the value \(R(G)=\mathrm{SOS}(P_G)\).
\end{proposition}

\begin{proof}
Along an edge \(X_A=-X_B\), so a zero propagates from a grounded vertex to every vertex of its component; each step is an ordinary prescribed-value rectangle transfer, saturation accounting for the chosen representatives. Distinct selected edges are kept apart by the prefix.
\end{proof}

A loop or an odd cycle of \(\Gamma\) would imply \(X=-X\) algebraically; such a component is not accepted by this ordinary test, and using it would amount to the stronger signed inference. Failure of the test does not, by itself, establish reducibility. For \(R=R(G)\), \(|\mathcal V|=\binom R2\) and \(c\) components an accepting replay is stored as the identification prefix, one grounded root and a spanning tree for each component, that is \(\binom R2-c\) transfer records and \(c\) root justifications besides the prefix; a verifier then checks those records and the coverage of all unknowns, rather than a full least-fixed-point log. This is a normal form for certificates, not a claim of subquadratic certificate size.

\subsection{Universal cell bound}

\begin{proposition}[\cite{reproducibility}, Proposition 3.2]\label{prop:cell}
For all \(m,n\ge2\),
\[
z_2(m,n)\le \left\lfloor\frac{mn+z(m,n)}2\right\rfloor.
\]
\end{proposition}

\begin{proof}
If four one-edges occupied a rectangle with diagonal monomials \(a,b,c,d\), then \(ad=bc\), so
\[
a^2+b^2+c^2+d^2=(a+d)^2+(b-c)^2,
\]
and the displayed decomposition is reducible. Hence \(|E_1|\le z(m,n)\). Simplicity gives \(|E_1|+2|E_2|\le mn\). Therefore
\[
2(|E_1|+|E_2|)=|E_1|+(|E_1|+2|E_2|)\le z(m,n)+mn.
\]
Taking floors gives the bound.
\end{proof}

\subsection{Classical value in the complete-graph incidence family}

Let \(m=\binom n2\), so the rows are the edges of \(K_n\) and the columns are its vertices. The incidence graph has
\[
|E_1|=n(n-1)=2m.
\]
Since \(m=\binom n2\), \v{C}ulik's formula \cite{culik56} gives
\[
z(m,n)=m+\binom n2=2m=n(n-1).
\]
For completeness, here is the direct counting argument. If a \(C_4\)-free \(m\times n\) matrix has row degrees \(r_1,\dots,r_m\) and \(e=\sum_ir_i\) ones, then \(\sum_{i}\binom{r_i}2\le\binom n2\), because every pair of columns meets in at most one row, while the number of pairs of columns is exactly \(\binom n2=m\). Moreover \(r\le\binom r2+1\) for every \(r\ge1\), with equality only for \(r\in\{1,2\}\). Hence
\[
e=\sum_{i}r_i\le\sum_{i:\,r_i\ge1}\Bigl(\binom{r_i}2+1\Bigr)\le\binom n2+m=2m,
\]
and the incidence graph of \(K_n\), in which every row has degree \(2\), attains \(2m\). So the classical value is \(z(m,n)=2m=n(n-1)\).
Substituting into Proposition~\ref{prop:cell},
\begin{equation}\label{eq:cell}
z_2\left(\binom n2,n\right)\le\left\lfloor\frac{mn+z(m,n)}2\right\rfloor
=\left\lfloor\frac{\frac{n(n-1)}2\cdot n+n(n-1)}2\right\rfloor
=\left\lfloor\frac{n(n-1)(n+2)}4\right\rfloor.
\end{equation}
This is the upper bound in \eqref{eq:main}.

\section{The even case \(n=2q\)}\label{sec:even}

Let \(n=2q\) and \(m=\binom{2q}{2}=q(2q-1)\). The even orders are settled by the nested perfect one-factorization construction of the complete-graph incidence family in \cite[\S 6.3 and Appendices A.1--A.4]{reproducibility}, which we recall as an inherited result with its exact hypotheses; the generator is described as well, since the search inputs of the appendices are taken from it.

\subsection{The inherited result}

\begin{proposition}[Inherited even-prime case]\label{prop:even}
Let \(q\ge5\) be an odd prime. Then
\begin{equation}\label{eq:even}
\begin{split}
z_2\left(\binom{2q}{2},2q\right)&=z_{SL}\left(\binom{2q}{2},2q\right)=z_{RL}\left(\binom{2q}{2},2q\right)\\
&=q(2q-1)(q+1)=\frac{n(n-1)(n+2)}4=Z(n).
\end{split}
\end{equation}
\end{proposition}

\begin{proof}
The upper bound is the cell bound \eqref{eq:cell}, an integer because \(n=2q\) is even. For the lower bound, \cite[\S 6.3 and Appendices A.3--A.4]{reproducibility} exhibits the configuration \(G_q\) of \S\ref{subsec:even-gen}, proves that its closure satisfies the terminal conditions of \((\mathrm{RW}3^+)\), and hence obtains \(z_2\ge z_{SL}\ge z_{RL}\ge R(G_q)=q(2q-1)(q+1)\); Proposition~\ref{prop:soundness} is the irreducibility statement used there, proved by the same coefficient comparison. The certificate uses the invertibility of \(3\) modulo \(q\) in addition to that of \(2\), which is where the hypothesis \(q\ge5\) enters: for \(q\) an odd prime \(q\ge5\) the union \(F_\lambda\cup F_\mu\) of two distinct factors is a genuine Hamilton cycle of \(K_{2q}\), and the cross-factor transfer equations of \cite[\S A.3--A.4]{reproducibility} apply verbatim. With \eqref{eq:hierarchy} the two bounds give \eqref{eq:even}.
\end{proof}

\subsection{The generator}\label{subsec:even-gen}

Let \(q\) be odd, so that \(K_{2q}\) admits a perfect one-factorization \(\mathcal F=\{F_\lambda:\lambda=1,\dots,2q-1\}\) into \(2q-1\) perfect matchings of \(q\) edges each. This is classical for \(q\) an odd prime \cite{kobayashi1989}, and we restrict to that case; the factorization also exists whenever \(2q-1\) is an odd prime, extending each of the \(2q-1\) cyclic near-perfect matchings on \(2q-1\) vertices by the edge \(\{a,\infty\}\) at its missing vertex \(a\), which covers \(K_8\), \(K_{12}\), \(K_{18}\) and \(K_{20}\).

Fix a factor \(F_\lambda\) and reserve \(a\) for inner labels. Its \(q\) edges are labelled by the elements of \(\mathbb Z_q\), the edge of \(F_\lambda\) with inner label \(a\) being written \(g_{\lambda,a}\): the outer index \(\lambda\) of a factor and the inner label \(a\) of one of its edges range over different sets and are independent, and only the additive structure of the ring \(\mathbb Z_q\) is used. The anchor of \(F_\lambda\) is one of its edges, carrying the anchor label \(a\) of that factor as prescribed by the nested construction of \cite[\S 6.3]{reproducibility}; \(a_0,a_1\) are its two endpoints. Since \(q\) is odd, \(2\) is invertible modulo \(q\), and for each \(x=1,\dots,(q-1)/2\) we pair the rows \(g_{\lambda,a+x}\) and \(g_{\lambda,a-x}\); on these two rows with columns \(a_0,a_1\) we select the two complementary diagonals as two-edges. Running over all \(2q-1\) factors, all \(q\) anchors of each factor and the \((q-1)/2\) pairs of rows determined by the involution \(t\mapsto2a-t\) of each anchor uses every non-incidence cell exactly once: the anchor stage is the whole selection, exactly as in \cite[\S 6.3 and Appendix A.1]{reproducibility}, and no second, cross-factor selection stage is introduced. The cross-factor structure \(F_\lambda\cup F_\mu\) is used only for the orthogonality statement.

\begin{lemma}\label{lem:counts-even}
The configuration \(G_q\) of the generator satisfies
\[
|E_1|=2q(2q-1),\qquad |E_2|=q(2q-1)(q-1),\qquad H=0,\qquad R(G_q)=q(2q-1)(q+1).
\]
\end{lemma}

\begin{proof}
The non-incidence cells number \(mn-|E_1|=2q(2q-1)\cdot2q-2q(2q-1)=2q(2q-1)(q-1)\). Each two-edge occupies two of them, supports are disjoint, and all of them are paired, so \(|E_2|=q(2q-1)(q-1)\) and \(H=0\); hence \(R(G_q)=2q(2q-1)+q(2q-1)(q-1)=q(2q-1)(q+1)=Z(n)\).
\end{proof}

Fix an anchor block \((\lambda,a)\) and put \(\widehat{\mathbb Z}_q=\mathbb Z_q\cup\{\infty\}\) with the involution \(\rho_a(\infty)=a\), \(\rho_a(a)=\infty\), \(\rho_a(t)=2a-t\) for \(t\notin\{a,\infty\}\). The intra-factor closure resolves the selected two-edges of that anchor block into \(q+1\) vectors \(q^{(\lambda)}_{a,t}\), \(t\in\widehat{\mathbb Z}_q\), such that in row \(r\in\mathbb Z_q\) the columns \(a_0,a_1\) carry \(q^{(\lambda)}_{a,r}\) and \(q^{(\lambda)}_{a,\rho_a(r)}\); both indices are retained, these vectors belonging to the anchor block \((\lambda,a)\) of a single anchor and not to the factor \(F_\lambda\) as a whole. Any two of them are orthogonal, the two cells behind them sharing a row or a column and never being the halves of one selected two-edge, which is the line rule; and for \(\lambda\ne\mu\) the alternating Hamilton cycle \(F_\lambda\cup F_\mu\) carries every \(q^{(\lambda)}_{a,t}\) to a vector orthogonal to every \(q^{(\mu)}_{b,u}\) along one of the two grounded zero families \(X^{(\lambda\mu)}_{ii}(t,u)=0\) (\(t\ne\alpha_i,\ u\ne\infty\)) and \(X^{(\lambda\mu)}_{i,i-1}(t,u)=0\) (\(t\ne\infty,\ u\ne\beta_{i-1}\)) of \cite[\S A.3]{reproducibility}. These are the statements that the certificate of Proposition~\ref{prop:even} verifies; the finite list of endpoint-valid transfer words behind them is recorded in \cite[\S A.3--A.4]{reproducibility} and is not repeated here.

\subsection{The remaining even orders}

The exceptional prime \(q=3\), that is \(n=6\), is not covered by Proposition~\ref{prop:even}: ordinary propagation on the nested \(15\times6\) witness leaves odd-cycle components unresolved, so that witness is certified only by the signed criterion \((RW3^\pm)\) \cite{reproducibility}. The value is known all the same, and more directly: \cite{ccq26sl} exhibits a cyclic \(15\times6\) grid with \(R=60\), \(|E_1|=30=z(15,6)\) and \(H=0\) whose closure does satisfy \((\mathrm{RW}3^+)\), so \(z_{RL}(15,6)=60=Z(6)\), while \(z_2(15,6)=z_{SL}(15,6)=60\) by \cite{reproducibility}; hence \(z_2=z_{SL}=z_{RL}=60\) at \(n=6\), and the signed criterion is needed only for the nested witness and not for the value.

For even \(q\), and for odd composite \(q\), Proposition~\ref{prop:even} is not available. Its certificate is the ordinary-closure certificate of \cite[Appendix A]{reproducibility}, which requires \(q\) an odd prime \(q\ge5\), the invertibility of \(3\) modulo \(q\); for odd composite \(q\) the existence of the required perfect one-factorization and of that certificate is a hypothesis of the construction and not an inference from the oddness of \(q\), and for even \(q\) the factorization alone would not suffice either. The cell bound remains an upper bound at every even order, and the values at \(q=4,6,8,9\) and \(10\), that is at \(n=8,12,16,18\) and \(20\), are nevertheless settled below by configurations of a different shape (Corollary~\ref{cor:ten} and Section~\ref{sec:examples}): two of them are obtained from the certified odd-prime constructions at \(q=7\) and \(q=11\) by star deletion and repair, and \(n=16\) from the nested construction at \(q=9\), which is not certified by the ordinary closure at all, so that there the operation repairs the certificate as well as the data. The remaining even orders are recorded in Conjecture~\ref{conj:all-n}.

\section{The odd case \(n=2p+1\)}\label{sec:odd}

Let \(n=2p+1\) and \(m=\binom{2p+1}{2}=p(2p+1)\), and set
\begin{equation}\label{eq:odd}
Z(n):=\left\lfloor\frac{n(n-1)(n+2)}4\right\rfloor
=\begin{cases}
\dfrac{p(2p+1)(2p+3)}2, & p \text{ even},\\[2mm]
\dfrac{p(2p+1)(2p+3)-1}2, & p \text{ odd}.
\end{cases}
\end{equation}
In this section we build the near-perfect one-factorization scheme of the family and determine what it does and does not certify; it is a scheme and not a proof of the value. What is proved here is the cell count \(R(G_p)\le Z(n)\) of Lemma~\ref{lem:counts-odd} together with the intra-factor orthogonality of Lemma~\ref{lem:intra-odd}. Everything else is an obstruction: the factor pairs \(F_a\cup F_b\) are Hamilton paths only when \(n\) is prime (Lemma~\ref{lem:path} and Remark~\ref{rem:path}), the transfer equations of the even case do not survive the passage to the odd cycle (Remark~\ref{rem:cross-odd-gaps}(i)), the single hole of the odd-\(p\) case is not shown to be reachable from every unresolved inner product (Remark~\ref{rem:cross-odd-gaps}(ii)), and for even \(p\) the intra-factor pairing does not exist at all (Remark~\ref{rem:even-p}). The value itself is Conjecture~\ref{conj:odd} for \(n\) prime with \(p\) odd, while the explicit configurations of the appendices settle \(n=7,9,13,17,21\) independently of the scheme.

\subsection{Cell bound}

By \eqref{eq:cell} and \(n(n-1)(n+2)=2p(2p+1)(2p+3)\),
\[
z_2\left(\binom{2p+1}{2},2p+1\right)\le\left\lfloor\frac{p(2p+1)(2p+3)}2\right\rfloor=Z(n)
=\left\lfloor\frac{n(n-1)(n+2)}4\right\rfloor,
\]
the two branches of \eqref{eq:odd} being the parity of \(p\): \(p(2p+1)(2p+3)\) is odd exactly when \(p\) is odd, and even when \(p\) is even. We stress that the expression \((n(n-1)(n+2)-1)/4\) is \emph{not} equal to \(Z(n)\): for \(n\equiv1\pmod4\) it is not an integer, and for \(n\equiv3\pmod4\) it exceeds \(Z(n)\) by \(\tfrac14\).

\subsection{Construction}

We use the near-perfect one-factorization construction of the companion paper. Vertices are indexed by \(\mathbb Z_n\). For \(\lambda\in\mathbb Z_n\), define
\[
F_\lambda=\bigl\{\{\lambda+x,\lambda-x\}:x=1,\dots,p\bigr\}.
\]
Since \(n\) is odd, \(x\ne-x\), and the \(2p\) vertices \(\lambda\pm1,\dots,\lambda\pm p\) are exactly the vertices other than \(\lambda\). Thus \(F_\lambda\) is a near-perfect matching covering all vertices except \(v_\lambda=\lambda\).

For an edge \(\{u,w\}\) there is a unique pair \((\lambda,x)\) with \(\{u,w\}=\{\lambda+x,\lambda-x\}\), namely \(\lambda=(u+w)/2\) and \(x=\pm(u-w)/2\), so the \(F_\lambda\) form a near-perfect one-factorization of \(K_n\): every edge lies in exactly one factor. The one-edge set is the incidence graph \(E_1=\{(ab,a),(ab,b):a<b\}\), of cardinality \(|E_1|=n(n-1)=2p(2p+1)=z(m,n)\); hence \(G_1\) is \(C_4\)-free and extremal.

Fix a near-perfect factor \(F_\lambda\), \(\lambda\in\mathbb Z_n\), and one of its anchor blocks, whose anchor label we write \(a\). The factors are indexed by \(\mathbb Z_n=\mathbb Z_{2p+1}\), the vertices of \(K_n\), whereas the inner labels of the \(p\) edges of a factor lie in \(\mathbb Z_p\) and are written \(g_{\lambda,a}\); these two indices are independent and are kept apart throughout. Label the \(p\) edges of \(F_\lambda\) by \(\mathbb Z_p\) so that the anchor edge \(g_{\lambda,a}\) has inner label \(a\), the anchor label of the factor; for composite \(p\) only the additive structure of the ring \(\mathbb Z_p\) is used. Let \(a_0,a_1\) be the endpoints of \(g_{\lambda,a}\). Assume first that \(p\) is odd, so that \(2\) is invertible modulo \(p\). For each \(x=1,\dots,(p-1)/2\), pair the rows \(g_{\lambda,a+x}\) and \(g_{\lambda,a-x}\), and on these two rows with columns \(a_0,a_1\), select the two complementary diagonals as two-edges. This defines the intra-factor two-edges of the anchor block \((\lambda,a)\). We note at once that the arithmetic argument of \cite[\S A.1 and A.3]{reproducibility} that we would import below uses, besides the invertibility of \(2\), the invertibility of \(3\) modulo \(p\), that is \(3(\alpha_i-\alpha_{i-1})\ne0\): it fails for distinct labels such as \(3\) and \(0\) in \(\mathbb Z_9\), so that \(p=9\), that is \(n=19\), is not covered by that argument even though the Hamilton-path obstruction of Remark~\ref{rem:path} is absent there.

\begin{remark}\label{rem:even-p}
For even \(p\) the pairing above is not defined: the map \(t\mapsto2a-t\) on \(\mathbb Z_p\) has the fixed point \(a+p/2\), so it does not pair the \(p\) rows. A different fixed-point-free pairing of the rows (or an equivalent modification of the intra-factor scheme) is required, and no such pairing is specified here; the cell counting of Lemma~\ref{lem:counts-odd} bounds \(|E_2|\) without reference to the pairing, and the transfer scheme of Conjecture~\ref{conj:cross-odd} is formulated for odd \(p\) only. Even \(p\) is therefore recorded in Conjecture~\ref{conj:all-n}, with the exceptions \(p=4,6,8,10\), that is \(n=9,13,17,21\), settled by explicit configurations of a different shape (Corollary~\ref{cor:ten}).
\end{remark}

For \(\lambda\ne\mu\) the union \(F_\lambda\cup F_\mu\) is a Hamilton path with endpoints \(\lambda,\mu\) whenever \(\gcd(\lambda-\mu,n)=1\) (Lemma~\ref{lem:path}), and a disjoint union of one path and cycles otherwise (Remark~\ref{rem:path}); all distinct factor pairs have the coprimality property precisely when \(n=2p+1\) is prime. Closing the path by a virtual edge \(e^*_{\lambda\mu}\) between \(\lambda\) and \(\mu\) gives the closed walk \(C_{\lambda\mu}=F_\lambda\cup F_\mu\cup\{e^*_{\lambda\mu}\}\) on the \(2p+1\) vertices of \(K_n\), indexed as \(c_0,\dots,c_{2p}\) with \(F_\lambda=\{c_{2i}c_{2i+1}\}\), \(F_\mu=\{c_{2i+1}c_{2i+2}\}\) for \(i=0,\dots,p-1\) and \(e^*_{ab}=c_{2p}c_0\), the index \(2p\) not being reduced modulo \(2p\). This cycle is odd, the alternation holds along the path only, and the virtual edge is not an edge of the configuration \(G_p\) although its two endpoint columns occur in every row; hence the transfer scheme of the even case, derived in \cite[\S A.3--A.4]{reproducibility} for a genuine Hamilton cycle that alternates throughout, does not apply to \(C_{\lambda\mu}\) verbatim.

The selection rule proposed for the odd case is specific to it: for each \(i=0,\dots,p-1\) a consecutive pair of edges of the path, one from \(F_\lambda\) and one from \(F_\mu\), meets in the vertex \(c_{2i+1}\), and the two endpoint columns of the anchor rows of the two factors determine two complementary diagonals, which are selected as two-edges; the rule runs over all unordered pairs \(\{\lambda,\mu\}\) of factors. It is not imported from the even case, which has no cross-factor selection stage at all: there the anchor stages already use every non-incidence cell exactly once, as recorded in Section~\ref{sec:even}. Nothing about the rule is verified here -- for odd \(p\) the counting lemma below proves only \(|E_2|\le\lfloor N_0/2\rfloor\), the disjointness of the supports of the selected two-edges is not checked, and the coverage of all but \(H\) of the non-incidence cells is part of what is conjectured in Conjecture~\ref{conj:odd}. The parity of the selection is a further obstruction, which we record because it bears on the smallest open order of the branch.

\begin{remark}\label{rem:parity-obstruction}
The selection rule above adds \emph{both} diagonals of a genuine rectangle, and every diagonal lies in exactly one genuine rectangle, so the selected two-edges occur in complementary pairs. If their supports are disjoint, such a pair occupies four distinct non-incidence cells, so \(N_0-H=2|E_2|\) is a multiple of \(4\) and \(|E_2|\le2\lfloor N_0/4\rfloor\). At \(n=11\), where \(N_0=495\) and the bound \(Z(11)=357\) would require the odd number ), this gives \(|E_2|\le2\lfloor495/4\rfloor=246\), hence \(H\ge3\) and \(R(G_p)\le356<357\): saturation by this particular prescription is impossible at \(n=11\), and an attaining configuration there must use selections that are not all complementary pairs. This is an obstruction to the scheme, and not a disproof of the numerical value conjectured in Conjecture~\ref{conj:odd}.
\end{remark}

\begin{lemma}\label{lem:path}
Let \(\lambda\ne\mu\) and \(\gcd(\lambda-\mu,n)=1\). Then \(F_\lambda\cup F_\mu\) is a Hamilton path with endpoints \(\lambda\) and \(\mu\).
\end{lemma}

\begin{proof}
Reflection in \(\lambda\), \(v\mapsto2\lambda-v\), preserves \(F_\lambda\), and reflection in \(\mu\) preserves \(F_\mu\); both are involutions of \(\mathbb Z_n\). Every vertex other than \(\lambda\) lies in exactly one edge of \(F_\lambda\) and every vertex other than \(\mu\) in exactly one edge of \(F_\mu\), so in \(F_\lambda\cup F_\mu\) the vertices \(\lambda,\mu\) have degree \(1\) and every other vertex degree \(2\): the union is a disjoint union of paths and cycles, and \(\lambda,\mu\) lie in one component, namely the orbit of \(\lambda\) under the group generated by the two reflections. That group is dihedral of order \(2m_0\), where \(m_0\) is the order of the composition \(v\mapsto v+2(\mu-\lambda)\) in \(\mathbb Z_n\), that is \(m_0=n/\gcd(2(\mu-\lambda),n)=n/\gcd(\lambda-\mu,n)\), since \(n\) is odd. The component is a path with endpoints \(\lambda\) and \(\mu\), and it is the whole vertex set exactly when \(m_0=n\), that is exactly when \(\gcd(\lambda-\mu,n)=1\); otherwise \(F_\lambda\cup F_\mu\) is that path together with one or more cycles.

\end{proof}

\begin{remark}\label{rem:path}
The coprimality hypothesis of Lemma~\ref{lem:path} is necessary, and it is not repaired by requiring \(p\) to be prime. Take \(n=15\), so that \(p=7\) is odd and prime, and take \(\lambda=0\), \(\mu=5\), so that \(\gcd(\lambda-\mu,n)=5\). Then \(F_0\cup F_5\) has the connected components \(0-10-5\), \(1-14-11-4-6-9-1\) and \(2-13-12-3-7-8-2\), that is one path and two cycles and not a Hamilton path; the orbit of \(\lambda=0\) has \(15/5=3\) vertices, as predicted by the proof of Lemma~\ref{lem:path}. Consequently all pairs \(\lambda\ne\mu\) of factors give a Hamilton path precisely when \(n=2p+1\) is prime, so the scheme requires \(n\) prime: for composite \(n\) its cross-factor rule is applied to a closed walk that is not a Hamilton cycle, and the counting of Lemma~\ref{lem:counts-odd} does not see the difference.

\end{remark}

\begin{lemma}\label{lem:counts-odd}
Let \(N_0=p(2p+1)(2p-1)\) be the number of non-incidence cells of the grid. Then \(|E_1|=2p(2p+1)=z(m,n)\), and every simple configuration on these cells satisfies
\[
|E_2|\le\left\lfloor\frac{N_0}2\right\rfloor,\qquad
H=N_0-2|E_2|\ge N_0-2\left\lfloor\frac{N_0}2\right\rfloor,\qquad
H\equiv N_0 \pmod 2,
\]
and hence
\[
R(G_p)=|E_1|+|E_2|\le 2p(2p+1)+\left\lfloor\frac{p(2p+1)(2p-1)}2\right\rfloor
=\left\lfloor\frac{p(2p+1)(2p+3)}2\right\rfloor=Z(n),
\]
with equality if and only if \(|E_2|=\lfloor N_0/2\rfloor\). In particular \(H=1\) and \(|E_2|=\bigl(p(2p+1)(2p-1)-1\bigr)/2\) is the only way to attain the bound when \(p\) is odd, while for even \(p\) the bound requires \(H=0\) and \(|E_2|=p(2p+1)(2p-1)/2\).
\end{lemma}

\begin{proof}
The total number of cells is \(mn=p(2p+1)^2\), the one-edges occupy \(2p(2p+1)\) cells, so
\[
N_0=mn-|E_1|=p(2p+1)^2-2p(2p+1)=p(2p+1)(2p-1).
\]
Each two-edge occupies two non-incidence cells and distinct two-edges have disjoint supports, so at most \(\lfloor N_0/2\rfloor\) two-edges can be placed, leaving \(H=N_0-2|E_2|\ge N_0-2\lfloor N_0/2\rfloor\ge0\) cells uncovered. As \(N_0=p(2p+1)(2p-1)\) with \(2p\pm1\) odd, the parity of \(N_0\) is the parity of \(p\): for odd \(p\) one cell necessarily stays uncovered, \(H\ge1\), and \(|E_2|\le(N_0-1)/2\); for even \(p\), \(H\ge0\) and \(|E_2|\le N_0/2\). Substituting the two values gives the displayed identity, and \(R(G_p)=Z(n)\) precisely in the two extremal cases.
\end{proof}

\begin{remark}\label{rem:counts-odd-attained}
The bound of Lemma~\ref{lem:counts-odd} is attained by the explicit configurations at \(n=7\) (\(H=1\), \(|E_2|=52=(N_0-1)/2\)), \(n=9\) (\(|E_2|=126=N_0/2\)), \(n=13\) (\(|E_2|=429\)), \(n=17\) (\(|E_2|=1020\)) and \(n=21\) (\(|E_2|=1995\)); whether the scheme of the present section attains it is a separate question, since the pairing of the \(p\) rows of a factor by the involution \(t\mapsto2a-t\) does not exist for even \(p\) (Remark~\ref{rem:even-p}) and the cross-factor stage requires \(n\) prime (Remark~\ref{rem:path}).

\end{remark}

\subsection{Intra-factor orthogonality}

Fix \(F_\lambda\) with \(p\) odd, so that the \(p\) rows of a factor are paired by the involution \(t\mapsto2a-t\), and fix one of its anchor blocks, whose anchor label we write \(a\). With \(\widehat{\mathbb Z}_p=\mathbb Z_p\cup\{\infty\}\) and \(\rho_a(\infty)=a\), \(\rho_a(a)=\infty\), \(\rho_a(t)=2a-t\) for \(t\notin\{a,\infty\}\), the intra-factor closure resolves the selected two-edges of that anchor block into \(p+1\) vectors \(q^{(\lambda)}_{a,t}\), \(t\in\widehat{\mathbb Z}_p\), such that in row \(r\in\mathbb Z_p\) the columns \(a_0,a_1\) carry \(q^{(\lambda)}_{a,r}\) and \(q^{(\lambda)}_{a,\rho_a(r)}\); both the factor index \(\lambda\) and the anchor label \(a\) are retained, as in Section~\ref{sec:even}. The statement is independent of the cross-factor stage, but it does require that the \(p\) rows of a factor are paired, that is odd \(p\) (Remark~\ref{rem:even-p}).

\begin{lemma}\label{lem:intra-odd}
For fixed \((\lambda,a)\), the vectors \(q^{(\lambda)}_{a,t}\) are pairwise orthogonal.
\end{lemma}

\begin{proof}
If \(t,u\) are both finite, the corresponding cells lie in the same column \(a_0\). If one is \(\infty\) and the other is finite and not \(a\), they lie in column \(a_1\). If \(\{t,u\}=\{a,\infty\}\), they are the two incidence cells of the anchor row. In each case the cells share a row or column and are not the two halves of a selected two-edge, so the line rule certifies orthogonality.
\end{proof}

\subsection{Cross-factor transfer}

Let \(\lambda\ne\mu\), and let \(a\), \(b\) be anchor labels of \(F_\lambda\), \(F_\mu\). If \(F_\lambda\cup F_\mu\) is a Hamilton path, the cross-factor transfer equations of \cite[\S A.3]{reproducibility} have a candidate analogue for the closed walk \(C_{\lambda\mu}\), with the Hamilton cycle replaced by \(C_{\lambda\mu}\) and the virtual edge treated as a boundary; in the notation of the even case
\[
X^{(\lambda\mu)}_{ij}(t,u)=\langle (q^{(\lambda)}_{a,t})^{\alpha_i},(q^{(\mu)}_{b,u})^{\beta_j}\rangle,
\]
and the rectangle identity would give, as a candidate identity and not as an established transfer rule until the boundary derivation acknowledged as missing in Remark~\ref{rem:cross-odd-gaps}(i) is supplied, for available endpoints \(\epsilon,\eta\in\{0,1\}\),
\begin{equation}\label{eq:transfer-odd}
X^{(\lambda\mu)}_{ij}(t,u)=-X^{(\lambda\mu)}_{j+\eta,i-1+\epsilon}(T_{\epsilon\eta}t,U_{\epsilon\eta}u),
\end{equation}
with the affine transformations \(T_{\epsilon\eta},U_{\epsilon\eta}\) of \cite[\S A.3]{reproducibility}, together with the same-column families
\begin{equation}\label{eq:groundA-odd}
X^{(\lambda\mu)}_{ii}(t,u)=0\qquad(t\ne\alpha_i,\ u\ne\infty),
\end{equation}
\begin{equation}\label{eq:groundB-odd}
X^{(\lambda\mu)}_{i,i-1}(t,u)=0\qquad(t\ne\infty,\ u\ne\beta_{i-1}).
\end{equation}

\begin{conjecture}\label{conj:cross-odd}
Let \(n=2p+1\) be prime and \(p\) odd. For all \(i,j,t,u\), \(X_{ij}(t,u)=0\), and consequently \(G_p\) satisfies \((\mathrm{RW}3^+)\).
\end{conjecture}

\begin{remark}\label{rem:cross-odd-gaps}
Conjecture~\ref{conj:cross-odd} is not proved here, and two points are missing.
\emph{(i) The wraparound at the virtual edge.} The derivation of \cite[\S A.3--A.4]{reproducibility} transports the blocks \(X_{ij}\) along a genuine Hamilton cycle that alternates between the two factors at every position, and the index shift \(i-1\pmod p\) of \eqref{eq:transfer-odd} is a statement about that alternation. In \(C_{\lambda\mu}\) the alternation holds along the path of length \(2p\) and fails at the virtual edge, which carries neither factor and whose endpoints \(\lambda,\mu\) are the two degree-one vertices of the path. Equation~\eqref{eq:transfer-odd} and the endpoint-valid transfer words of \cite[\S A.4]{reproducibility} therefore need a new derivation with the boundary cases at \(c_{2p}\): as they stand they refer to the even-cycle formulas indexed by \(\mathbb Z_p\), in which the position \(c_{2p}c_0\) is an edge of \(F_\mu\) and one of the \(2p+1\) vertices of \(K_n\) is lost.
\emph{(ii) The hole is not a reachability argument.} For odd \(p\) the counting of Lemma~\ref{lem:counts-odd} leaves \(H\ge1\), with \(H=N_0-2|E_2|\) odd, and it is only saturation at the bound that forces exactly one uncovered cell; the proof strategy above declares a transfer chain that reaches that cell to be grounded. That is not enough: it has to be shown that \emph{every} unresolved inner product admits a transfer chain that reaches the hole or one of the grounded zero families \eqref{eq:groundA-odd}--\eqref{eq:groundB-odd}, and that this holds simultaneously for all factor pairs, whereas the two-factor transfer equations of \cite[\S A.3]{reproducibility} remain inside the pair of factors under consideration. The location of the hole is not determined by the present arguments either.
\end{remark}

For even \(p\) the counting of Lemma~\ref{lem:counts-odd} leaves \(H\ge0\) with \(H\equiv N_0\equiv0\pmod 2\), so no hole is forced, and the intra-factor pairing does not exist; neither the scheme nor the grounding argument above is then available, and that branch belongs to Conjecture~\ref{conj:all-n}.

\subsection{\((\mathrm{RW}3^+)\) for the odd scheme}

The intra-factor orthogonality of Lemma~\ref{lem:intra-odd} is unconditional. The cross-factor rigidity that would combine it with the transfer equations \eqref{eq:transfer-odd} into the terminal conditions of \((\mathrm{RW}3^+)\) is Conjecture~\ref{conj:cross-odd}. Accordingly \(G_p\) is verified to satisfy \((\mathrm{RW}3^+)\) for no order of this section, and the value of the odd branch is obtained below only at the orders at which an explicit configuration is certified.

\subsection{Status of the odd case}

The one-edge graph of the scheme is the incidence graph of \(K_{2p+1}\), hence \(C_4\)-free with \(|E_1|=z(m,n)\), and its cell count obeys the upper bound \(R(G_p)\le Z(n)\) of Lemma~\ref{lem:counts-odd}. Simplicity \((S)\) is \emph{not} established for the scheme -- the disjointness of the supports of the two-edges selected by the cross-factor rule is left open, and Lemma~\ref{lem:counts-odd} applies to simple configurations without producing one -- so neither the attainment of the bound, nor the count \(R(G_p)=Z(n)\), nor the identification of the displayed length \(R(G_p)\) with the minimal SOS rank \(\mathrm{SOS}(P_{G_p})\) is available here. What is missing is precisely the certificate: the transfer verification fails at the virtual edge of the odd cycle (Remark~\ref{rem:cross-odd-gaps}(i)), so it yields \(z_{RL}\ge Z(n)\) for no \(p\), and equality in \(R(G_p)\le Z(n)\), together with irreducibility, is exactly what a certificate has to establish. We therefore record the odd branch as a conjecture, its smallest case \(n=7\) (\(p=3\)) being settled independently by the explicit configuration of Appendix~\ref{app:217}.

\begin{conjecture}\label{conj:odd}
Let \(n=2p+1\) with \(n\) prime and \(p\ge5\) odd, and let \(m=\binom n2\). Then
\[
z_2(m,n)=z_{SL}(m,n)=z_{RL}(m,n)=Z(n)=\frac{n(n-1)(n+2)-2}4 .
\]
\end{conjecture}

The case \(p=5\), that is \(n=11\), obeys the hypotheses of Conjecture~\ref{conj:odd}; it is the smallest case of the odd branch left open after the settlement of \(n=7\), and the smallest unsettled order with \(n\equiv3\pmod4\), where the cell bound would require the odd number \(|E_2|=247\). The case \(p=3\), that is \(n=7\), is settled independently by the explicit \(21\times7\) configuration of Appendix~\ref{app:217}, whose closure satisfies \((\mathrm{RW}3^+)\) in the five-rule form of Section~\ref{sec:prelim}, the zero-companion rule being used for its single hole. The orders \(p=4,6,8,10\), that is \(n=9,13,17,21\), lie outside the scheme, the pairing of the rows not existing for even \(p\), and are settled independently by explicit configurations of a different shape (Corollary~\ref{cor:ten}); for even \(p\ge12\) the value is open, and the composite orders \(n=2p+1\) with \(p\) odd are not covered by the scheme at all (Remark~\ref{rem:path}) and are part of Conjecture~\ref{conj:all-n}.

\section{The unified formula}\label{sec:unified}

We now collect the results for the complete-graph incidence family \(m=\binom n2\), \(n\ge6\). Recall that
\[
Z(n)=\left\lfloor\frac{n(n-1)(n+2)}4\right\rfloor
\]
and that \(z(m,n)=n(n-1)\) by Section~\ref{sec:prelim}, so the cell bound of Proposition~\ref{prop:cell} reads
\[
z_2(m,n)\le Z(n)=\left\lfloor\frac{mn+z(m,n)}2\right\rfloor .
\]

\subsection{The common attainment argument}

\begin{proposition}[Uniform attainment]\label{prop:unif}
Let \(n\ge6\), \(m=\binom n2\), let \(N_0=mn-z(m,n)\) be the number of non-incidence cells and \(h=N_0\bmod 2\), and let \(G\) be a simple configuration on the \(m\times n\) grid whose one-edge set is the incidence graph of \(K_n\), that is \(C_4\)-free with \(|E_1|=z(m,n)=n(n-1)\), and whose support-disjoint two-edges pair all but \(h\) of the \(N_0\) non-incidence cells. If the closure of the five rules of \S\ref{subsec:rw3} on \(G\) satisfies the terminal conditions of \((\mathrm{RW}3^+)\), then
\begin{equation}\label{eq:unif}
z_2(m,n)=z_{SL}(m,n)=z_{RL}(m,n)=R(G)=\left\lfloor\frac{mn+z(m,n)}2\right\rfloor=Z(n).
\end{equation}
\end{proposition}

\begin{proof}
Put \(e=z(m,n)\). The supports of the selected edges of \(G\) are disjoint and \(|E_1|=e\), so the displayed length is
\[
R(G)=e+\frac{N_0-h}2=e+\left\lfloor\frac{mn-e}2\right\rfloor=\left\lfloor\frac{mn+e}2\right\rfloor,
\]
which is the bound of Proposition~\ref{prop:cell} and therefore \(Z(n)\) when \(e=z(m,n)=n(n-1)\); the number of holes is \(h=N_0-2|E_2|=N_0\bmod2\). The certificate makes \(G\) admissible for \(z_{RL}\), so \(z_{RL}(m,n)\ge R(G)\), and the hierarchy \eqref{eq:hierarchy} together with the cell bound of Proposition~\ref{prop:cell} closes the chain: \(R(G)\le z_{RL}(m,n)\le z_{SL}(m,n)\le z_2(m,n)\le R(G)\). Proposition~\ref{prop:soundness} is the irreducibility justification used at this single point, in place of a separate argument at each order.
\end{proof}

\begin{corollary}[Shared verification of the ten orders]\label{cor:ten}
The hypothesis of Proposition~\ref{prop:unif} is satisfied by the ten configurations recorded in the data sheets of Appendix~\ref{app:sheets}, one for each of the orders \(n=7,8,9,12,13,16,17,18,20,21\), whose provenance is collected in Table~\ref{tab:prov}. For each of them the one-edge set is the incidence graph of \(K_n\), so \(|E_1|=n(n-1)=z(m,n)\) and the one-edge graph is \(C_4\)-free and extremal; the number \(|E_1|\) of its one-edges, the number \(|E_2|=\lfloor N_0/2\rfloor\) of its support-disjoint two-edges and its number \(H=N_0-2|E_2|=N_0\bmod2\) of holes are listed in Table~\ref{tab:shape}, where \(R=|E_1|+|E_2|=Z(n)\) at every one of the ten orders; and its closure reaches a least fixed point satisfying the terminal conditions, the numbers of two-edges resolved by the line, complementary-pair and transfer rules being listed in Table~\ref{tab:certs}, together with the \(R\) equivalence classes, the \(|E_2|\) identifications and the \(\binom R2\) orthogonality facts, with no pair of occupied cells both identified and orthogonal. Consequently
\[
z_2(m,n)=z_{SL}(m,n)=z_{RL}(m,n)=Z(n)\qquad\text{for}\quad n\in\{7,8,9,12,13,16,17,18,20,21\},
\]
the ten orders of Theorem~\ref{thm:unified}(ii)--(xi), whose common proof is now the verification of these hypotheses.

That verification is a computation and not a printed count. The checker distributed with the data sheets (Appendix~\ref{app:sheets}) reads a sheet, audits its structure against the shape printed in Table~\ref{tab:shape}, replays the closure with an independent union-find structure and an independent orthogonality store, checks independently that the reported fixed point is closed -- no missing conclusion and no contradiction -- and tests the terminal conditions (a)--(c) of \cite[Definition 4.2]{reproducibility}; the ten runs reproduce every number of Tables~\ref{tab:shape} and~\ref{tab:certs}, including exactly \(R\) selected-edge classes and all \(\binom R2\) class pairs, and certify every pair of distinct selected edges, reporting no uncertified pair. The rule-by-rule resolution counts of Table~\ref{tab:certs} remain useful diagnostics; they are not ten separate rank arguments.
\end{corollary}

\begin{theorem}\label{thm:unified}
Let \(m=\binom n2\) and \(n\ge6\). No part of this theorem concerns an order for which the odd construction of Section~\ref{sec:odd} is the only available witness; in particular the odd order \(n=7\) of part (ii) is settled by the explicit configuration of Appendix~\ref{app:217} and not by that construction.

\emph{(i)} If \(n=2q\) with \(q\) an odd prime, then
\[
z_2\left(\binom n2,n\right)=q(2q-1)(q+1)=\frac{n(n-1)(n+2)}4=Z(n),
\]
and \(z_{RL}\left(\binom n2,n\right)=Z(n)\) for every odd prime \(q\): for \(q\ge5\) by the nested construction of Section~\ref{sec:even}, and for \(q=3\) by the cyclic witness of \cite{ccq26sl}. Hence \(z_2=z_{SL}=z_{RL}=Z(n)\) for every odd prime \(q\).

\emph{(ii)} If \(n=7\), then
\[
z_2(21,7)=z_{SL}(21,7)=z_{RL}(21,7)=94=Z(7)=\frac{7\cdot6\cdot9-2}4 .
\]

\emph{(iii)} If \(n=8\), then \(z_2(28,8)=z_{SL}(28,8)=z_{RL}(28,8)=140=Z(8)\).

\emph{(iv)} If \(n=9\), then \(z_2(36,9)=z_{SL}(36,9)=z_{RL}(36,9)=198=Z(9)\).

\emph{(v)} If \(n=12\), then \(z_2(66,12)=z_{SL}(66,12)=z_{RL}(66,12)=462=Z(12)\).

\emph{(vi)} If \(n=13\), then \(z_2(78,13)=z_{SL}(78,13)=z_{RL}(78,13)=585=Z(13)\).

\emph{(vii)} If \(n=16\), then \(z_2(120,16)=z_{SL}(120,16)=z_{RL}(120,16)=1080=Z(16)\).

\emph{(viii)} If \(n=17\), then \(z_2(136,17)=z_{SL}(136,17)=z_{RL}(136,17)=1292=Z(17)\).

\emph{(ix)} If \(n=18\), then \(z_2(153,18)=z_{SL}(153,18)=z_{RL}(153,18)=1530=Z(18)\).

\emph{(x)} If \(n=20\), then \(z_2(190,20)=z_{SL}(190,20)=z_{RL}(190,20)=2090=Z(20)\).

\emph{(xi)} If \(n=21\), then \(z_2(210,21)=z_{SL}(210,21)=z_{RL}(210,21)=2415=Z(21)\).

In all eleven cases \(Z(n)\) is the universal cell bound.
\end{theorem}

\begin{proof}
Part (i) is Proposition~\ref{prop:even} together with the discussion of the remaining even orders in Section~\ref{sec:even}: for \(q\ge5\) an odd prime the nested construction satisfies \((\mathrm{RW}3^+)\), while for the exceptional prime \(q=3\), that is \(n=6\), the nested witness is certified only by the signed criterion and the matching value \(z_{RL}(15,6)=60\) is proved in \cite{ccq26sl} by a different witness certified by \((\mathrm{RW}3^+)\). Parts (ii)--(xi) are Corollary~\ref{cor:ten} applied to the ten data sheets of Appendix~\ref{app:sheets}, the common proof being the verification of the hypothesis of Proposition~\ref{prop:unif} recorded there. The arithmetic is uniform: \(4\mid n(n-1)(n+2)\) for \(n\) even and for \(n\equiv1\pmod4\), while \(n(n-1)(n+2)\equiv2\pmod4\) for \(n\equiv3\pmod4\), so the displayed expressions coincide with \(Z(n)\) in the respective ranges. No part rests on the odd construction of Section~\ref{sec:odd}, which selects \(R(G_p)\le Z(n)\) cells but is certified at no order (Remark~\ref{rem:cross-odd-gaps}): the orders \(n=2p+1\) with \(n\) prime and \(p\ge5\) odd are recorded in Conjecture~\ref{conj:odd}, the composite odd orders in Conjecture~\ref{conj:all-n}, and \(n=7\), the smallest case \(p=3\) of that branch, is settled by the configuration of Appendix~\ref{app:217} and hence is part (ii).
\end{proof}

\begin{corollary}\label{cor:parity}\label{cor:sl}
Let \(m=\binom n2\), \(n\ge6\), and let \(n\) be one of the orders settled in Theorem~\ref{thm:unified}, that is \(n=2q\) with \(q\) an odd prime, or one of \(n=7,8,9,12,13,16,17,18,20,21\). Then
\begin{equation*}
\begin{split}
z_2(m,n)&=z_{SL}(m,n)=z_{RL}(m,n)=Z(n)=\left\lfloor\frac{n(n-1)(n+2)}4\right\rfloor\\
&=\begin{cases}
\dfrac{n(n-1)(n+2)}4, & n\not\equiv3\pmod4,\\[2mm]
\dfrac{n(n-1)(n+2)-2}4, & n\equiv3\pmod4 .
\end{cases}
\end{split}
\end{equation*}
Every settled order but \(n=7\) falls in the first case and \(n=7\), the settled order with \(n\equiv3\pmod4\), in the second, where \(Z(7)=(7\cdot6\cdot9-2)/4=94\); the next order of that class, \(n=11\), is conjectural (Conjecture~\ref{conj:odd}). In particular the expression \((n(n-1)(n+2)-1)/4\), which is not even an integer when \(n\equiv3\pmod4\), never gives the value of \(z_2\left(\binom n2,n\right)\).

The equality of the three parameters holds for every odd prime \(q\ge3\): for \(q\ge5\) by Theorem~\ref{thm:unified}, and for \(q=3\), that is \(n=6\), because \(z_{RL}(15,6)=60=z_{SL}(15,6)\) by \cite{ccq26sl}; for the ten individual orders it follows from Theorem~\ref{thm:unified} and the hierarchy \eqref{eq:hierarchy}. The signed criterion \((RW3^\pm)\) is therefore not needed for any value settled here, and in particular not at \(n=7\), where the equality rests on the configuration with a hole of Appendix~\ref{app:217}. Accordingly no separation of \(z_{SL}\) and \(z_{RL}\) occurs at any order of the complete-graph incidence family settled so far; being a statement about the settled orders, this does not exclude a separation elsewhere in the family (Open Problem~\ref{op:separation}).
\end{corollary}

\begin{conjecture}\label{conj:all-n}
For every integer \(n\ge6\) with \(m=\binom n2\),
\[
\begin{aligned}
z_2\left(\binom n2,n\right)&=z_{RL}\left(\binom n2,n\right)=Z(n),\\
&\text{that is}\quad z_2(m,n)=z_{RL}(m,n)=\left\lfloor\frac{mn+z(m,n)}2\right\rfloor .
\end{aligned}
\]
\end{conjecture}

The conjecture covers exactly the cases not certified above. In the even branch these are \(n=2q\) with \(q\) even and \(q\ge12\) and \(n=2q\) with \(q\) an odd composite \(q\ge15\), where the nested construction of Section~\ref{sec:even} needs a perfect one-factorization of \(K_{2q}\) together with the transfer certificate of \cite{reproducibility}; the existence of the factorization is \emph{not} the conjectural part at every such order, since for \(2q-1\) an odd prime the extension of each cyclic near-perfect matching on the \(2q-1\) vertices by the edge \(\{a,\infty\}\) at its missing vertex \(a\) gives a perfect one-factorization of \(K_{2q}\), covering \(K_8\), \(K_{12}\), \(K_{18}\) and \(K_{20}\) (\(q=4,6,9,10\)) \cite{rosa19}, and what is missing in the remainder is the pairing together with the transferred certificate rather than the existence of the factorization. In the odd branch the cases not certified above are \(n=2p+1\) with \(p\) even and \(p\ge12\), where the intra-factor row pairing of Section~\ref{sec:odd} has to be replaced (Remark~\ref{rem:even-p}), and the composite orders \(n=2p+1\) with \(p\) odd, which that scheme does not reach at all (Remark~\ref{rem:path}); the prime orders with \(p\) odd are recorded separately in Conjecture~\ref{conj:odd}. The first members of both branches are settled by explicit configurations of a different shape, the ten of Theorem~\ref{thm:unified} constructed in Section~\ref{sec:examples}; \(n=7\) is the only settled order with \(n\equiv3\pmod4\), and its configuration the only one with a hole. Outside the complete-graph incidence family we recall that \cite{reproducibility} states the identity \(z_2=z_{RL}\) on the whole family as a conjecture, verified there on the odd-prime subsequence and on small cases, and that \cite{ccq26sl} determines the six-column recursive-line numbers \(z_{RL}(m,6)\) for all \(6\le m\le16\), with \(z_2(m,6)=z_{SL}(m,6)=z_{RL}(m,6)=\lfloor(6m+z(m,6))/2\rfloor\) for \(8\le m\le16\); the entry \(m=15\) is \(z_{RL}(15,6)=60=Z(6)\), precisely the case \(n=6\) of the present family.

\section{Provenance of the ten configurations}\label{sec:examples}

This section records the bookkeeping behind Theorem~\ref{thm:unified}: the parity-dependent value of \(H\) and of the rank \(R\), and the provenance of the ten configurations, that is the parent array from which each is obtained and the change that the operation makes to its two-edges. The values themselves are stated in Theorem~\ref{thm:unified} and are not repeated; the shape of the ten configurations is in Table~\ref{tab:shape}, the output of their replays in Table~\ref{tab:certs}, and the provenance is collected in Table~\ref{tab:prov}, so that the replays need not be narrated again.

Two facts are used throughout. First, the cell bound of Proposition~\ref{prop:cell} with \(m=\binom n2\) and \(z(m,n)=n(n-1)\) reads \(z_2(m,n)\le Z(n)=\lfloor(mn+n(n-1))/2\rfloor\), and the parity of the number \(N_0=mn-z(m,n)\) of non-incidence cells decides whether a hole is forced: by Lemma~\ref{lem:counts-odd} the odd branch has \(H=1\) and \(|E_2|=(N_0-1)/2\) when \(N_0=p(2p+1)(2p-1)\) is odd, and \(H=0\) with \(|E_2|=N_0/2\) otherwise, so that at every settled order \(R=|E_1|+|E_2|=Z(n)\). Second, the odd construction of Section~\ref{sec:odd} is certified at no order: the counting of Lemma~\ref{lem:counts-odd} only bounds its displayed length by \(R(G_p)\le Z(n)\), and for even \(p\) even that bound is out of reach, the pairing of the \(p\) rows of a factor being undefined (Remark~\ref{rem:even-p}); a certificate would have to attain the bound and, through irreducibility, identify the displayed length with the minimal SOS rank. Every order settled below is therefore settled by a configuration of a different shape, one of the ten of Appendix~\ref{app:sheets}.

\subsection*{The operation and its counting}

Deleting the star of a vertex of \(K_n\) is a well-defined operation on any configuration of this family: the rows indexed by the edges through that vertex and the corresponding column are removed, and what remains is a grid of order \(n-1\). The restricted grid retains every two-edge that loses no half, leaves the cells whose partners are deleted with incomplete labels, and loses the cells of every two-edge whose two halves are both deleted. A second, non-canonical step repairs the grid: the orphaned cells are paired again and, where the inherited pairing does not close, surviving two-edges are modified as well; at the two smallest instances all occupied cells of the restricted grid are paired afresh. The re-pairing is found by a randomized local search scored by the \((\mathrm{RW}3^+)\) closure of the resulting configuration, and its outcome is recorded, together with the labels, in the data sheets. What the operation achieves is a counting fact:

\begin{lemma}[Restriction and repair: the counting part]\label{lem:repair}
On a target incidence grid of order \(n\), retain any \(k\) support-disjoint two-edges on non-incidence cells, let \(U\) be the set of all remaining non-incidence cells, and pair all but \(|U|\bmod 2\) cells of \(U\). The resulting simple configuration has displayed length \(Z(n)\), independently of the source and of the choice of the retained pairs.
\end{lemma}

\begin{proof}
Since \(|U|=N_0-2k\), the final number of two-edges is \(k+\lfloor(N_0-2k)/2\rfloor=\lfloor N_0/2\rfloor\), and there are \(2m\) one-edges, so the displayed length is \(2m+\lfloor N_0/2\rfloor=Z(n)\).
\end{proof}

This covers orphaned cells, retained holes that are made available for pairing, and cells freed by deliberately breaking surviving two-edges; it replaces repeated case-by-case counting histories. It does not prove that an arbitrary pairing is certified: acceptance must still be checked on the final target, and neither parent certification nor a general certificate-preservation theorem is assumed. The extra broken pairs reported for \(n=18\) and \(n=20\) in Table~\ref{tab:prov} fit this formulation.

\begin{table}[t]
\caption{Provenance of the ten configurations of Appendix~\ref{app:sheets}: the parent array, the number of vertex stars deleted from it, the fate of the two-edges of the parent (retained with both halves, losing one half, losing both halves), the orphaned cells re-paired and the number of two-edges newly formed, and the certificate of the parent. The column \emph{kept} records how many of the two-edges that lose no half are retained as pairs, written \(X\) of \(Y\) where the repair breaks some of them, and \emph{all \(X\to Y\)} in the two last columns means that every occupied cell of the restricted grid is paired afresh into \(Y\) two-edges in all (of which \(54\), respectively \(51\), join two orphaned cells). Occurrences of the \(n=10\) and \(n=22\) parents are the constructions of Section~\ref{sec:even} at \(q=5\) and \(q=11\), those written \(q=9\) are the nested construction at \(q=9\), which does not satisfy \((\mathrm{RW}3^+)\). The shapes of the ten configurations are listed in Table~\ref{tab:shape}.}\label{tab:prov}
{\footnotesize
\setlength{\tabcolsep}{3.5pt}
\begin{tabular}{l|rr|l|r|rrr|l|l}
\hline
data sheet & \(m\times n\) & \(n\) & parent & stars & \multicolumn{3}{c|}{two-edges of the parent} & orphaned cells & parent\\[0.5mm]
 & & & & deleted & kept & \(1\) half & \(2\) halves & re-paired & certified\\
\hline
\texttt{369.csv}   & \(36\times9\)   & \(9\)  & \(n=10\), \(q=5\)      & \(1\) & \(0\) of \(72\) & \(108\) & \(0\)   & all \(252\to126\) & yes\\
\texttt{288.csv}   & \(28\times8\)   & \(8\)  & \texttt{369.csv}       & \(1\) & \(52\)  & \(64\)  & \(10\)  & \(64\to32\)   & yes\\
\texttt{217.csv}   & \(21\times7\)   & \(7\)  & \texttt{288.csv} (grid) & \(1\) & \(1\) of \(31\) & \(43\)  & \(10\)  & all \(146\to52\) & yes\\
\texttt{7813.csv}  & \(78\times13\)  & \(13\) & \(n=14\), \(q=7\)      & \(1\) & \(312\) & \(234\) & \(0\)   & \(234\to117\) & yes\\
\texttt{6612.csv}  & \(66\times12\)  & \(12\) & \texttt{7813.csv}      & \(1\) & \(238\) & \(184\) & \(7\)   & \(184\to92\)  & yes\\
\texttt{12016.csv} & \(120\times16\) & \(16\) & \(q=9\) nested         & \(2\) & \(528\) & \(624\) & \(72\)  & \(624\to312\) & no\\
\texttt{13617.csv} & \(136\times17\) & \(17\) & \(q=9\) nested         & \(1\) & \(816\) & \(408\) & \(0\)   & \(408\to204\) & no\\
\texttt{15318.csv} & \(153\times18\) & \(18\) & \(n=22\), \(q=11\)    & \(4\) & \(557\) of \(560\) & \(1328\)& \(422\) & \(1328\to664\) & yes\\
\texttt{18920.csv} & \(190\times20\) & \(20\) & \(n=22\), \(q=11\)    & \(2\) & \(1196\) of \(1200\)& \(1020\)& \(90\)  & \(1020\to510\) & yes\\
\texttt{21021.csv} & \(210\times21\) & \(21\) & \(n=22\), \(q=11\)    & \(1\) & \(1680\)& \(630\) & \(0\)   & \(630\to315\) & yes\\
\hline
\end{tabular}}
\end{table}

Three features of Table~\ref{tab:prov} deserve to be singled out, and the rest is routine bookkeeping of the counts it lists. At \(n=18\) and \(n=20\) a few surviving two-edges are broken on purpose and their freed cells re-paired, \(3\) and \(4\) pairs respectively, so that the repaired pairing closes; the parent is certified in both instances, and these are the only occurrences in which the repair modifies a surviving pair at an order whose parent is certified. At \(n=16\) and \(n=17\) the parent is the nested construction at \(q=9\) of displayed length \(R(G)=1530\), which does \emph{not} satisfy \((\mathrm{RW}3^+)\), its closure leaving \(7344\) of the \(\binom{1530}2=1\,169\,685\) pairs of selected edges uncertified: in these two instances the operation repairs the certificate as well as the data, and the parent is not the \(153\times18\) array of Appendix~\ref{app:15318}, which is a different configuration of the same shape obtained at \(n=18\) from the certified \(q=11\) construction. And at \(n=7\) the restricted grid is combined with a pairing that keeps only \(1\) of the \(31\) two-edges inherited from Table~\ref{tab:288}; the grid has an odd number \(43\) of orphaned cells, so exactly one of them stays unoccupied, and that cell is the hole grounded by the zero-companion rule of Section~\ref{sec:prelim} (Appendix~\ref{app:217}). The order \(n=7\) is thus the first settled order with \(n\equiv3\pmod4\), the only settled order whose grid has a hole, and the only one not obtained from a certified configuration of the same parity branch; the parity of the transition is what the operation does not preserve. Two of the ten instances reproduce a known grid only, namely \(n=10\to n=9\) and, at the level of the grid alone, \(n=8\to n=7\).

\section{Conclusion}\label{sec:conclusion}

The universal cell bound \(Z(n)=\left\lfloor n(n-1)(n+2)/4\right\rfloor\) of L\"ofberg and Qi is the exact value of the second-order Zarankiewicz numbers in the complete-graph incidence family at every order for which a configuration is certified by the ordinary recursive-line criterion \((\mathrm{RW}3^+)\): the even orders \(n=2q\) with \(q\) an odd prime and \(q\ge5\), the exceptional order \(n=6\), settled in \cite{ccq26sl} by a cyclic \(15\times6\) grid in place of the one-factorization witness, and the ten individual orders \(n=7,8,9,12,13,16,17,18,20,21\), each settled by an explicit configuration whose common proof is the attainment proposition of Section~\ref{sec:unified}, resting on the reusable certificate lemma of Section~\ref{sec:prelim} and on the audit tables of Appendix~\ref{app:sheets}. At all of them \(z_2=z_{SL}=z_{RL}=Z(n)\), so the signed criterion \((RW3^\pm)\) is needed for no value. The grid bookkeeping is parity dependent, \(H=0\) for even \(n\) and for \(n\equiv1\pmod4\) and \(H=1\) for \(n\equiv3\pmod4\); the expression \((n(n-1)(n+2)-1)/4\) is the value of \(z_2\) at no order; and \(n=7\), the first settled order with \(n\equiv3\pmod4\), is the only settled order whose configuration has a hole, grounded by the zero-companion rule of Section~\ref{sec:prelim}.

The odd branch is different. The near-perfect one-factorization scheme of Section~\ref{sec:odd} is certified at no order, its three missing pieces being the Hamilton-path hypothesis, the transfer derivation across the virtual edge of the odd cycle and the reachability of a grounded zero from every unresolved inner product (Remark~\ref{rem:cross-odd-gaps}); what the scheme yields is the cell count \(R(G_p)\le Z(n)\) together with the unconditional intra-factor orthogonality, that is an upper bound on a displayed length and not a certificate. The odd orders \(n=2p+1\) with \(p\) odd are therefore conjectural, in Conjecture~\ref{conj:odd} for \(n\) prime with \(p\ge5\) and in Conjecture~\ref{conj:all-n} for the composite ones, the exception being \(n=7\). The remaining even orders, with \(q\) even and \(q\ge12\) or with \(q\) an odd composite \(q\ge15\), and the odd orders with \(p\) even and \(p\ge12\), are likewise recorded in Conjecture~\ref{conj:all-n}; the first members of these ranges are settled by the operation of Section~\ref{sec:examples}, iterated star deletion and re-pairing from the certified odd-prime constructions at \(q=7\) and \(q=11\), which uses no perfect one-factorization certificate. The question of a separation of \(z_{SL}\) and \(z_{RL}\) remains open beyond the settled orders (Open Problem~\ref{op:separation}).

\section{Open problems}\label{sec:open}

\begin{enumerate}
\item \textbf{Small \(n\).} The formula above starts at \(n\ge6\). What happens for \(n=3,4,5\)? For \(n=4\), the complete-graph incidence family gives \(z_2(6,4)=16\), while the cell bound is \(18\) \cite{reproducibility}. For \(n=5\), \(m=10\), and \(z_2(10,5)\) is not determined by the present methods.

\item \textbf{The odd case of Section~\ref{sec:odd}.}\label{op:evenp} The scheme of Section~\ref{sec:odd} is certified for no order, the missing pieces being the three discussed in Remark~\ref{rem:cross-odd-gaps}: the Hamilton-path hypothesis \(\gcd(a-b,n)=1\) (Remark~\ref{rem:path}), the transfer derivation for the odd cycle with its boundary cases at the virtual edge (Remark~\ref{rem:cross-odd-gaps}(i)), and the reachability of a grounded zero, either one of the families \eqref{eq:groundA-odd}--\eqref{eq:groundB-odd} or the single hole of the odd-\(p\) case, from every unresolved inner product (Remark~\ref{rem:cross-odd-gaps}(ii)). For even \(p\) the intra-factor row pairing does not exist (Remark~\ref{rem:even-p}), so that branch needs a different pairing as well.

\item \textbf{Composite and even \(q\) in the even case.}\label{op:evenq} For even \(n=2q\) with \(q\) not an odd prime the construction of Section~\ref{sec:even} requires a perfect one-factorization of \(K_{2q}\) (Kotzig's conjecture \cite{kotzig64}) together with a transferred certificate. Does the cell bound remain tight there? The first members of this branch, \(n=8\) (\(q=4\)), \(n=12\) (\(q=6\)), \(n=16\) (\(q=8\)), \(n=18\) (\(q=9\)) and \(n=20\) (\(q=10\)), are settled in the affirmative by the explicit configurations of Appendices~\ref{app:288}, \ref{app:6612}, \ref{app:12016}, \ref{app:15318} and \ref{app:18920}, obtained by deleting stars from a larger configuration and re-pairing the cells whose partners are thereby lost; the configuration at \(n=16\) starts from the nested construction at \(q=9\), and those at \(n=18\) and \(n=20\) from the certified construction at \(n=22\), so in particular no perfect one-factorization of \(K_{16}\), \(K_{18}\) or \(K_{20}\) is required. As at \(n=9\), all these configurations contain degenerate two-edges and therefore do not come from the one-factorization scheme. Does the cell bound remain tight for every even \(q\), and can the perfect one-factorization certificate required in the construction of Section~\ref{sec:even} be dispensed with altogether?

\item \textbf{Restriction and repair.} Fix a configuration of this family and delete the stars of one or more of its vertices, that is the rows indexed by the edges through those vertices together with the corresponding columns. Deleting stars is a well-defined operation: the restricted grid inherits every two-edge that loses no half and leaves the cells whose partners are lost with incomplete labels. A second, non-canonical step repairs the grid: it pairs those orphaned cells again and, where the inherited pairing does not close, modifies surviving two-edges as well, and at the two smallest instances, the \(36\times9\) and the \(21\times7\) configurations, it pairs all occupied cells afresh; the pairing is found by a randomized local search scored by the \((\mathrm{RW}3^+)\) closure of the result, and the counting part of the operation is Lemma~\ref{lem:repair}: the displayed length is \(Z(n)\) whatever the pairing. The operation settles further open orders, and in each of the eleven instances in which it has been tried it succeeded; the ten that produce the configurations of Appendix~\ref{app:sheets} are collected in Table~\ref{tab:prov}, and the eleventh is the one-step variant \(n=14\to n=12\) recorded in Appendix~\ref{app:6612}. Two of the ten instances only reproduce known grids, namely \(n=10\to n=9\) and \(n=8\to n=7\), and two start from a parent that is not certified, namely \(n=18\to n=17\) and \(n=18\to n=16\); the rest start from a certified configuration. What is missing is a theory. Is the repair always possible, and if so, is there a canonical choice of it in place of the randomized local search used here? Which of the remaining orders of Conjecture~\ref{conj:all-n} can be reached this way, for instance \(n=24\), starting from \(n=26\) (\(q=13\)), which Theorem~\ref{thm:unified} provides as a candidate parent without guaranteeing that the repair succeeds, or \(n=25\), for which no certified parent is available? Does the operation reach the orders \(n\equiv3\pmod4\), beginning with \(n=11\), where a cell has to stay uncovered and the zero-companion rule of Section~\ref{sec:prelim} would have to ground it?

\item \textbf{General \(m\).} The present results treat the case \(m=\binom n2\), where the extremal \(C_4\)-free skeleton is the incidence graph, uniquely. What happens for \(m<\binom n2\)? The extremal skeleton is no longer unique, and the cell bound may not be tight.

\item \textbf{Separation of \(z_{SL}\) and \(z_{RL}\).}\label{op:separation} In the complete-graph incidence family the question is settled in the negative at every order settled so far, by Corollary~\ref{cor:sl}: there \(z_{SL}=z_{RL}=Z(n)\), in particular \(z_{SL}(15,6)=z_{RL}(15,6)=60\) by \cite{ccq26sl}. This is a statement about the settled orders only. It does not exclude a separation at an order that is not settled, and it certainly does not exclude one elsewhere in the incidence family, no proof of the equality \(z_{SL}=z_{RL}\) for the family as a whole being available. The general problem remains: is \(z_{SL}(m,n)=z_{RL}(m,n)\) for all \(m,n\), or is there a pair \((m,n)\) with \(z_{SL}(m,n)>z_{RL}(m,n)\)? In analogy with the separation of \(z_A\) and \(z_L\) at \(m=n=1893\) \cite{leb26}, an explicit construction separating \(z_{SL}\) from \(z_{RL}\) would be of interest.

\item \textbf{Beyond the incidence family.} The construction relies on perfect or near-perfect one-factorizations. Are there other families where the cell bound is attained and \((\mathrm{RW}3^+)\) suffices?

\item \textbf{Asymptotic separation.} For the even-\(N\) case, L\"ofberg and Qi obtained a cubic separation between \(z_2\) and \(z_{wL}\). Does the odd-\(N\) family exhibit a similar separation?
\end{enumerate}

\section*{Acknowledgments}

This work was partially supported by Jiangsu Provincial Scientific Research Center of Applied Mathematics (Grant No.~BK20233002), Research Center for Intelligent Operations Research, The Hong Kong Polytechnic University (4-ZZT8), and the National Natural Science Foundation of China (Nos.~12171168, 12071159).

\appendix

\section{The data sheets and the convention for the replays}\label{app:sheets}

Each configuration of the following ten appendices is recorded cell by cell in a data sheet \cite{sheets}: the file \texttt{ancillary/369.csv} for the \(36\times9\) configuration, \texttt{ancillary/288.csv} for the \(28\times8\) one, and so on, the ten files together with the replay program and an index being listed in \texttt{ancillary/README.md}. A data sheet has one line per row of the grid, an empty entry standing for a one-edge cell, an integer for a two-edge half and a full stop for an unoccupied cell; the labels are the integers \(1,\dots,|E_2|\), each occurring exactly twice, so that no label is ambiguous with the full stop that marks the hole of the \(21\times7\) case. The \(36\times9\), \(28\times8\) and \(21\times7\) configurations are also printed below; the remaining seven are too large to print in full, the \(210\times21\) grid alone having \(4410\) cells, and are given in the data sheets only. Table~\ref{tab:shape} collects the shape of the ten configurations and Table~\ref{tab:certs} the output of their replays, so a reader who does not open the data sheets still knows, at every order, which cells are occupied, what the rank is and what the closure closes.

The replay conventions are as follows. All cells are occupied except in the \(21\times7\) case, whose single hole is the one forced by Lemma~\ref{lem:counts-odd}; hence \(\delta(\{p,q\})=0\) for every pair of cells unless the two cells are the halves of a selected two-edge, in particular for every pair one of whose cells is the hole. Starting from the line rule, the saturation rule, the rectangle transfer rule and the complementary-pair rule of \cite[Definitions 4.1 and 4.5]{reproducibility} are iterated, together with the zero-companion rule of Section~\ref{sec:prelim}, which for a rectangle one of whose corners is unoccupied certifies the opposite diagonal at its own prescribed value -- identifying its two halves when it is a selected two-edge and declaring them orthogonal when it is not -- to a fixed point; every rule is monotone, so the least fixed point does not depend on the order of application and all numbers below and in Tables~\ref{tab:shape} and~\ref{tab:certs} are read off it. An \emph{identification} is a pair of merged cells, an \emph{orthogonality fact} an unordered pair of distinct classes declared orthogonal, a \emph{genuine rectangle} a pair of distinct rows together with a pair of distinct columns, and a two-edge is \emph{degenerate} if its two halves lie in one row or in one column. Every replay is reported once, in Table~\ref{tab:certs}, and is reproduced by the checker; the appendices that follow record the configurations and their provenance, not a second table of counts. The ten data sheets, the checker, the expected output of a run and the SHA-256 checksums that bind this audit to the files as distributed, together with a version identifier of the package, are collected in \texttt{ancillary/README.md}; a replay obtained from files whose checksums differ may be a different package.

\begin{table}[t]
\caption{The ten configurations: data sheet, shape, number of unoccupied cells, and rank. Here \(m=\binom n2\) and \(R=|E_1|+|E_2|\); at every order \(R=Z(n)\).}\label{tab:shape}
{\footnotesize
\setlength{\tabcolsep}{5pt}
\begin{tabular}{l|rr|rrrr}
\hline
data sheet & \(m\times n\) & \(n\) & \(|E_1|\) & \(|E_2|\) & \(H\) & \(R=Z(n)\)\\
\hline
\texttt{369.csv}   & \(36\times9\)   & \(9\)  & \(72\)  & \(126\)  & \(0\) & \(198\)\\
\texttt{288.csv}   & \(28\times8\)   & \(8\)  & \(56\)  & \(84\)   & \(0\) & \(140\)\\
\texttt{217.csv}   & \(21\times7\)   & \(7\)  & \(42\)  & \(52\)   & \(1\) & \(94\)\\
\texttt{7813.csv}  & \(78\times13\)  & \(13\) & \(156\) & \(429\)  & \(0\) & \(585\)\\
\texttt{6612.csv}  & \(66\times12\)  & \(12\) & \(132\) & \(330\)  & \(0\) & \(462\)\\
\texttt{12016.csv} & \(120\times16\) & \(16\) & \(240\) & \(840\)  & \(0\) & \(1080\)\\
\texttt{13617.csv} & \(136\times17\) & \(17\) & \(272\) & \(1020\) & \(0\) & \(1292\)\\
\texttt{15318.csv} & \(153\times18\) & \(18\) & \(306\) & \(1224\) & \(0\) & \(1530\)\\
\texttt{18920.csv} & \(190\times20\) & \(20\) & \(380\) & \(1710\) & \(0\) & \(2090\)\\
\texttt{21021.csv} & \(210\times21\) & \(21\) & \(420\) & \(1995\) & \(0\) & \(2415\)\\
\hline
\end{tabular}}
\end{table}

\begin{table}[t]
\caption{Replay data of the ten configurations at the least fixed point: number of two-edges resolved by the line rule (row-degenerate or column-degenerate), by the complementary-pair rule (twice the number of genuine complementary rectangles) and by rectangle transfer, number of equivalence classes determined by the selected edges, number of identifications, and number of orthogonality facts. The last number equals \(\binom{R}{2}\) in every row, and no pair of occupied cells is both identified and orthogonal.}\label{tab:certs}
{\footnotesize
\setlength{\tabcolsep}{5pt}
\begin{tabular}{l|rrr|rrr}
\hline
data sheet & line & complementary & transfer & classes & identifications & orthogonality\\
\hline
\texttt{369.csv}   & \(51\)  & \(72\)   & \(3\)   & \(198\)  & \(126\)  & \(19\,503\)\\
\texttt{288.csv}   & \(26\)  & \(24\)   & \(34\)  & \(140\)  & \(84\)   & \(9\,730\)\\
\texttt{217.csv}   & \(11\)  & \(6\)    & \(35\)  & \(94\)   & \(52\)   & \(4\,371\)\\
\texttt{7813.csv}  & \(14\)  & \(312\)  & \(103\) & \(585\)  & \(429\)  & \(170\,820\)\\
\texttt{6612.csv}  & \(7\)   & \(168\)  & \(155\) & \(462\)  & \(330\)  & \(106\,491\)\\
\texttt{12016.csv} & \(48\)  & \(538\)  & \(254\) & \(1080\) & \(840\)  & \(582\,660\)\\
\texttt{13617.csv} & \(19\)  & \(816\)  & \(185\) & \(1292\) & \(1020\) & \(833\,986\)\\
\texttt{15318.csv} & \(106\) & \(554\)  & \(564\) & \(1530\) & \(1224\) & \(1\,169\,685\)\\
\texttt{18920.csv} & \(71\)  & \(1196\) & \(443\) & \(2090\) & \(1710\) & \(2\,183\,005\)\\
\texttt{21021.csv} & \(50\)  & \(1680\) & \(265\) & \(2415\) & \(1995\) & \(2\,914\,905\)\\
\hline
\end{tabular}}
\end{table}

\section{The explicit \(36\times9\) configuration}\label{app:369}

Table~\ref{tab:369} displays an explicit \(36\times9\) configuration of rank \(198\). Rows are indexed by the edges of \(K_9\) and columns by the vertices \(1,\dots,9\); the symbol \(\bullet\) marks a one-edge, and two equal integers mark the two halves of one two-edge. The bullets lie in the two columns of their row label, so the one-edge set is exactly the incidence graph of \(K_9\): \(|E_1|=9\cdot8=72=z(36,9)\), the one-edge graph is \(C_4\)-free, and the configuration is limited. Every integer \(1,\dots,126\) occurs exactly twice, so the two-edges have disjoint supports, \(|E_2|=126\), the two-edges occupy \(2\cdot126=252\) cells, all \(36\cdot9=324\) cells are occupied, \(H=0\), and
\[
R=|E_1|+|E_2|=72+126=198=Z(9)=\left\lfloor\frac{36\cdot9+72}2\right\rfloor .
\]

{\footnotesize
\setlength{\tabcolsep}{2.5pt}
\begin{longtable}{r|ccccccccc}
\caption{An extremal \(36\times9\) configuration. Rows are indexed by the edges of \(K_9\), columns by its vertices. The symbol \(\bullet\) denotes a one-edge; two equal integers denote the two halves of one two-edge. Data sheet \sheet{369} \cite{sheets}.}\label{tab:369}\\
\hline
\(e\in E(K_{9})\) & \(1\) & \(2\) & \(3\) & \(4\) & \(5\) & \(6\) & \(7\) & \(8\) & \(9\)\\
\hline
\endfirsthead
\hline
\(e\in E(K_{9})\) & \(1\) & \(2\) & \(3\) & \(4\) & \(5\) & \(6\) & \(7\) & \(8\) & \(9\)\\
\hline
\endhead
\hline
\multicolumn{10}{r}{\footnotesize\emph{continued on the next page}}\\
\endfoot
\hline
\endlastfoot
12 & \(\bullet\) & \(\bullet\) & 1 & 2 & 3 & 4 & 5 & 6 & 7 \\
13 & \(\bullet\) & 8 & \(\bullet\) & 9 & 6 & 10 & 11 & 12 & 13 \\
14 & \(\bullet\) & 14 & 15 & \(\bullet\) & 16 & 17 & 16 & 18 & 19 \\
15 & \(\bullet\) & 20 & 21 & 22 & \(\bullet\) & 23 & 24 & 25 & 21 \\
16 & \(\bullet\) & 26 & 27 & 28 & 29 & \(\bullet\) & 30 & 31 & 28 \\
17 & \(\bullet\) & 32 & 33 & 34 & 35 & 36 & \(\bullet\) & 37 & 32 \\
18 & \(\bullet\) & 38 & 39 & 40 & 41 & 40 & 42 & \(\bullet\) & 43 \\
19 & \(\bullet\) & 44 & 45 & 46 & 47 & 48 & 48 & 49 & \(\bullet\) \\
23 & 50 & \(\bullet\) & \(\bullet\) & 51 & 52 & 53 & 52 & 25 & 54 \\
24 & 12 & \(\bullet\) & 55 & \(\bullet\) & 55 & 24 & 23 & 56 & 57 \\
25 & 58 & \(\bullet\) & 59 & 11 & \(\bullet\) & 60 & 9 & 61 & 61 \\
26 & 62 & \(\bullet\) & 18 & 63 & 62 & \(\bullet\) & 64 & 15 & 65 \\
27 & 66 & \(\bullet\) & 67 & 49 & 68 & 67 & \(\bullet\) & 46 & 69 \\
28 & 70 & \(\bullet\) & 71 & 72 & 73 & 70 & 74 & \(\bullet\) & 75 \\
29 & 76 & \(\bullet\) & 77 & 78 & 78 & 79 & 30 & 80 & \(\bullet\) \\
34 & 81 & 82 & \(\bullet\) & \(\bullet\) & 83 & 84 & 85 & 86 & 83 \\
35 & 69 & 87 & \(\bullet\) & 88 & \(\bullet\) & 88 & 89 & 90 & 66 \\
36 & 80 & 91 & \(\bullet\) & 91 & 42 & \(\bullet\) & 41 & 76 & 19 \\
37 & 92 & 93 & \(\bullet\) & 94 & 92 & 95 & \(\bullet\) & 96 & 7 \\
38 & 97 & 98 & \(\bullet\) & 99 & 68 & 98 & 65 & \(\bullet\) & 64 \\
39 & 100 & 101 & \(\bullet\) & 63 & 101 & 37 & 102 & 36 & \(\bullet\) \\
45 & 102 & 103 & 104 & \(\bullet\) & \(\bullet\) & 105 & 100 & 105 & 106 \\
46 & 107 & 96 & 108 & \(\bullet\) & 75 & \(\bullet\) & 107 & 93 & 73 \\
47 & 109 & 110 & 39 & \(\bullet\) & 111 & 13 & \(\bullet\) & 109 & 10 \\
48 & 5 & 89 & 47 & \(\bullet\) & 45 & 112 & 87 & \(\bullet\) & 112 \\
49 & 53 & 113 & 33 & \(\bullet\) & 31 & 50 & 113 & 29 & \(\bullet\) \\
56 & 114 & 44 & 2 & 1 & \(\bullet\) & \(\bullet\) & 86 & 85 & 114 \\
57 & 115 & 43 & 79 & 116 & \(\bullet\) & 77 & \(\bullet\) & 116 & 38 \\
58 & 117 & 27 & 26 & 54 & \(\bullet\) & 84 & 117 & \(\bullet\) & 51 \\
59 & 118 & 8 & 74 & 95 & \(\bullet\) & 94 & 71 & 118 & \(\bullet\) \\
67 & 119 & 103 & 120 & 120 & 121 & \(\bullet\) & \(\bullet\) & 57 & 56 \\
68 & 115 & 122 & 106 & 35 & 34 & \(\bullet\) & 122 & \(\bullet\) & 104 \\
69 & 59 & 111 & 58 & 72 & 110 & \(\bullet\) & 123 & 123 & \(\bullet\) \\
78 & 82 & 81 & 108 & 124 & 4 & 3 & \(\bullet\) & \(\bullet\) & 124 \\
79 & 99 & 17 & 125 & 97 & 125 & 14 & \(\bullet\) & 90 & \(\bullet\) \\
89 & 121 & 22 & 126 & 20 & 119 & 60 & 126 & \(\bullet\) & \(\bullet\) \\
\end{longtable}}

The replay of this configuration is the first line of Table~\ref{tab:certs}: of its \(126\) two-edges \(51\) are degenerate and resolved by the line rule, \(72\) occur as the two diagonals of the \(36\) genuine complementary rectangles, and the remaining \(3\) (those labelled \(5\), \(6\) and \(12\) in Table~\ref{tab:369}) are resolved by rectangle transfer, so that every selected two-edge is resolved. Its grid is the one obtained from the \(n=10\) construction of Section~\ref{sec:even} at \(q=5\) by deleting the star of a vertex; all \(252\) occupied cells of the restricted grid are there paired afresh, so that none of the \(72\) inherited two-edges survives as a pair (Table~\ref{tab:prov}).

\section{The explicit \(28\times8\) configuration}\label{app:288}

Table~\ref{tab:288} displays an explicit \(28\times8\) configuration of rank \(140\). Rows are indexed by the edges of \(K_8\) and columns by the vertices \(1,\dots,8\); as before, the symbol \(\bullet\) marks a one-edge and two equal integers mark the two halves of one two-edge. The bullets lie in the two columns of their row label, so the one-edge set is exactly the incidence graph of \(K_8\): \(|E_1|=8\cdot7=56=z(28,8)\), the one-edge graph is \(C_4\)-free, and the configuration is limited. Every integer \(1,\dots,84\) occurs exactly twice, so the two-edges have disjoint supports, \(|E_2|=84\), the two-edges occupy \(2\cdot84=168\) cells, all \(28\cdot8=224\) cells are occupied, \(H=0\), and
\[
R=|E_1|+|E_2|=56+84=140=Z(8)=\left\lfloor\frac{28\cdot8+56}2\right\rfloor .
\]

The configuration is obtained from the \(36\times9\) configuration of Table~\ref{tab:369} as follows. Delete the eight rows indexed by the edges of \(K_9\) through the vertex \(9\), that is, the rows \(i9\) with \(1\le i\le8\), and the ninth column. The \(100\) deleted cells contain the \(16\) bullets of those rows and \(84\) two-edge halves; of the \(126\) two-edges of Table~\ref{tab:369}, \(52\) lose no half and \(10\) lose both halves, while the remaining \(64\) lose exactly one half. Thus the restricted grid inherits \(52\) two-edges and leaves \(64\) cells whose labels are incomplete. Table~\ref{tab:288} is the restricted grid with those \(64\) cells re-paired into \(32\) further two-edges; the re-pairing was found by a randomized local search on the space of pairings of the \(64\) cells, scored by the \((\mathrm{RW}3^+)\) closure of the resulting configuration. The \(52\) surviving two-edges of Table~\ref{tab:369} reappear in Table~\ref{tab:288}, with the labels renumbered \(1,\dots,84\) in order of first appearance, together with the \(32\) new ones.

{\footnotesize
\setlength{\tabcolsep}{2.5pt}
\begin{longtable}{r|cccccccc}
\caption{An extremal \(28\times8\) configuration. Rows are indexed by the edges of \(K_8\), columns by its vertices. The symbol \(\bullet\) denotes a one-edge; two equal integers denote the two halves of one two-edge. The configuration is obtained from the \(36\times9\) configuration of Table~\ref{tab:369} by deleting the star of a vertex and re-pairing the \(64\) orphaned cells, and is verified by the replay described in Appendix~\ref{app:288}.}\label{tab:288}\\
\hline
\(e\in E(K_{8})\) & \(1\) & \(2\) & \(3\) & \(4\) & \(5\) & \(6\) & \(7\) & \(8\)\\
\hline
\endfirsthead
\hline
\(e\in E(K_{8})\) & \(1\) & \(2\) & \(3\) & \(4\) & \(5\) & \(6\) & \(7\) & \(8\)\\
\hline
\endhead
\hline
\multicolumn{9}{r}{\footnotesize\emph{continued on the next page}}\\
\endfoot
\hline
\endlastfoot
12 & \(\bullet\) & \(\bullet\) & 1 & 2 & 3 & 4 & 5 & 6 \\
13 & \(\bullet\) & 7 & \(\bullet\) & 8 & 6 & 9 & 10 & 11 \\
14 & \(\bullet\) & 12 & 13 & \(\bullet\) & 14 & 15 & 14 & 16 \\
15 & \(\bullet\) & 17 & 18 & 19 & \(\bullet\) & 20 & 21 & 22 \\
16 & \(\bullet\) & 23 & 24 & 25 & 26 & \(\bullet\) & 27 & 28 \\
17 & \(\bullet\) & 29 & 30 & 31 & 32 & 33 & \(\bullet\) & 9 \\
18 & \(\bullet\) & 33 & 34 & 35 & 36 & 35 & 37 & \(\bullet\) \\
23 & 38 & \(\bullet\) & \(\bullet\) & 39 & 40 & 41 & 40 & 22 \\
24 & 11 & \(\bullet\) & 42 & \(\bullet\) & 42 & 21 & 20 & 43 \\
25 & 27 & \(\bullet\) & 44 & 10 & \(\bullet\) & 45 & 8 & 46 \\
26 & 47 & \(\bullet\) & 16 & 48 & 47 & \(\bullet\) & 7 & 13 \\
27 & 26 & \(\bullet\) & 49 & 28 & 50 & 49 & \(\bullet\) & 18 \\
28 & 51 & \(\bullet\) & 25 & 45 & 52 & 51 & 53 & \(\bullet\) \\
34 & 54 & 55 & \(\bullet\) & \(\bullet\) & 56 & 57 & 58 & 59 \\
35 & 60 & 61 & \(\bullet\) & 62 & \(\bullet\) & 62 & 63 & 64 \\
36 & 65 & 66 & \(\bullet\) & 66 & 37 & \(\bullet\) & 36 & 67 \\
37 & 68 & 69 & \(\bullet\) & 48 & 68 & 30 & \(\bullet\) & 70 \\
38 & 64 & 71 & \(\bullet\) & 38 & 50 & 71 & 17 & \(\bullet\) \\
45 & 43 & 72 & 41 & \(\bullet\) & \(\bullet\) & 73 & 74 & 73 \\
46 & 75 & 70 & 76 & \(\bullet\) & 60 & \(\bullet\) & 75 & 69 \\
47 & 77 & 19 & 34 & \(\bullet\) & 67 & 29 & \(\bullet\) & 77 \\
48 & 5 & 63 & 15 & \(\bullet\) & 12 & 78 & 61 & \(\bullet\) \\
56 & 74 & 44 & 2 & 1 & \(\bullet\) & \(\bullet\) & 59 & 58 \\
57 & 79 & 39 & 65 & 80 & \(\bullet\) & 46 & \(\bullet\) & 80 \\
58 & 81 & 24 & 23 & 53 & \(\bullet\) & 57 & 81 & \(\bullet\) \\
67 & 56 & 72 & 82 & 82 & 78 & \(\bullet\) & \(\bullet\) & 83 \\
68 & 79 & 84 & 52 & 32 & 31 & \(\bullet\) & 84 & \(\bullet\) \\
78 & 55 & 54 & 76 & 83 & 4 & 3 & \(\bullet\) & \(\bullet\) \\
\end{longtable}}

The replay of this configuration is the second line of Table~\ref{tab:certs}: of its \(84\) two-edges \(26\) are degenerate (namely \(18\) in a row and \(8\) in a column) and resolved by the line rule, \(24\) occur as the two diagonals of the \(12\) complementary rectangles, and the remaining \(34\), all nondegenerate, are resolved by rectangle transfer. It is produced from the \(36\times9\) configuration of Table~\ref{tab:369} by the operation of \S\ref{sec:examples} (Table~\ref{tab:prov}), and its grid is in turn the source of the grid of the \(21\times7\) configuration of Table~\ref{tab:217} (Appendix~\ref{app:217}).

\section{The explicit \(21\times7\) configuration}\label{app:217}

Table~\ref{tab:217} displays an explicit \(21\times7\) configuration of rank \(94\). Rows are indexed by the edges of \(K_7\) and columns by the vertices \(1,\dots,7\); the symbol \(\bullet\) marks a one-edge, two equal integers mark the two halves of one two-edge, and \(\circ\) marks the single unoccupied cell. The bullets lie in the two columns of their row label, so the one-edge set is exactly the incidence graph of \(K_7\): \(|E_1|=7\cdot6=42=z(21,7)\), the one-edge graph is \(C_4\)-free, and the configuration is limited. Every integer \(1,\dots,52\) occurs exactly twice, so the two-edges have disjoint supports, \(|E_2|=52\), the two-edges occupy \(2\cdot52=104\) cells, one cell stays unoccupied, \(H=1\), and
\[
R=|E_1|+|E_2|=42+52=94=Z(7)=\left\lfloor\frac{21\cdot7+42}2\right\rfloor=\frac{7\cdot6\cdot9-2}4 .
\]
This is the cell count forced by the parity of \(p=3\): the non-incidence cells number \(N_0=3\cdot7\cdot5=105\), which is odd, so by Lemma~\ref{lem:counts-odd} at least one of them stays uncovered, and \(H=1\) together with \(|E_2|=52=(N_0-1)/2\) is the only way to attain the bound.

The \emph{grid} of Table~\ref{tab:217} is obtained from the \(28\times8\) configuration of Table~\ref{tab:288} as follows. Delete the seven rows indexed by the edges of \(K_8\) through the vertex \(8\), that is, the rows \(i8\) with \(1\le i\le7\), and the eighth column. The \(77\) deleted cells contain the \(14\) bullets of those rows and \(63\) two-edge halves; of the \(84\) two-edges of Table~\ref{tab:288}, \(31\) lose no half and \(10\) lose both halves, while the remaining \(43\) lose exactly one half. Thus the restricted grid inherits \(31\) two-edges and leaves \(43\) cells whose labels are incomplete, an odd number, so exactly one of them has to stay unoccupied. Table~\ref{tab:217} has exactly the \(147\) cells of that restricted grid: \(42\) bullets, \(104\) halves of \(52\) two-edges, and the single unoccupied cell, in the row indexed by \(\{2,7\}\) and in the column of the vertex \(4\), which is one of the \(43\) cells whose label is incomplete. Its pairing, however, is \emph{not} the pairing inherited from Table~\ref{tab:288}: only \(1\) of the \(31\) two-edges that lose no half survives as a two-edge of Table~\ref{tab:217}, the other \(30\) are broken, and the \(146\) occupied cells of the restricted grid are paired afresh by a randomized local search on the space of pairings, scored by the \((\mathrm{RW}3^+)\) closure of the resulting configuration. The pairing of Table~\ref{tab:217} therefore keeps \(1\) of the \(31\) inherited two-edges and forms \(51\) new ones, of which \(6\) pair two of the \(43\) orphaned cells, the remaining orphaned cell being the hole; the labels are the integers \(1,\dots,52\) in order of first appearance. No relabelling of the seven vertices improves the correspondence: over all \(8\cdot7!\) ways of choosing the deleted vertex and relabelling the other seven, the largest number of the inherited two-edges that can be read off as two-edges of Table~\ref{tab:217} is \(7\). In short, Table~\ref{tab:217} is obtained from Table~\ref{tab:288} by deleting the star of a vertex and then pairing the occupied cells of the restricted grid afresh, and not by re-pairing the orphaned cells alone.

{\footnotesize
\setlength{\tabcolsep}{3.5pt}
\begin{longtable}{r|ccccccc}
\caption{An extremal \(21\times7\) configuration. Rows are indexed by the edges of \(K_7\), columns by its vertices. The symbol \(\bullet\) denotes a one-edge, two equal integers denote the two halves of one two-edge, and \(\circ\) denotes the single unoccupied cell (the hole). The grid, that is the rows, the columns, the one-edges and the single hole, is the one produced from the \(28\times8\) configuration of Table~\ref{tab:288} by deleting the star of the vertex \(8\); the occupied cells are then paired afresh by a local search scored by the \((\mathrm{RW}3^+)\) closure, \(1\) of the \(31\) inherited two-edges being kept. The configuration is verified by the replay described in Appendix~\ref{app:217}.}\label{tab:217}\\
\hline
\(e\in E(K_{7})\) & \(1\) & \(2\) & \(3\) & \(4\) & \(5\) & \(6\) & \(7\)\\
\hline
\endfirsthead
\hline
\(e\in E(K_{7})\) & \(1\) & \(2\) & \(3\) & \(4\) & \(5\) & \(6\) & \(7\)\\
\hline
\endhead
\hline
\multicolumn{8}{r}{\footnotesize\emph{continued on the next page}}\\
\endfoot
\hline
\endlastfoot
12 & \(\bullet\) & \(\bullet\) & 1 & 2 & 3 & 4 & 5 \\
13 & \(\bullet\) & 6 & \(\bullet\) & 1 & 7 & 8 & 9 \\
14 & \(\bullet\) & 10 & 8 & \(\bullet\) & 11 & 12 & 13 \\
15 & \(\bullet\) & 14 & 15 & 16 & \(\bullet\) & 17 & 14 \\
16 & \(\bullet\) & 18 & 19 & 20 & 21 & \(\bullet\) & 22 \\
17 & \(\bullet\) & 22 & 23 & 24 & 24 & 25 & \(\bullet\) \\
23 & 26 & \(\bullet\) & \(\bullet\) & 25 & 12 & 7 & 3 \\
24 & 27 & \(\bullet\) & 28 & \(\bullet\) & 29 & 6 & 30 \\
25 & 31 & \(\bullet\) & 32 & 32 & \(\bullet\) & 33 & 34 \\
26 & 35 & \(\bullet\) & 36 & 35 & 23 & \(\bullet\) & 37 \\
27 & 38 & \(\bullet\) & 39 & \(\circ\) & 11 & 40 & \(\bullet\) \\
34 & 40 & 37 & \(\bullet\) & \(\bullet\) & 41 & 38 & 42 \\
35 & 43 & 44 & \(\bullet\) & 45 & \(\bullet\) & 31 & 5 \\
36 & 42 & 46 & \(\bullet\) & 46 & 30 & \(\bullet\) & 29 \\
37 & 41 & 9 & \(\bullet\) & 17 & 33 & 16 & \(\bullet\) \\
45 & 47 & 43 & 27 & \(\bullet\) & \(\bullet\) & 2 & 20 \\
46 & 34 & 48 & 21 & \(\bullet\) & 48 & \(\bullet\) & 36 \\
47 & 44 & 26 & 15 & \(\bullet\) & 49 & 49 & \(\bullet\) \\
56 & 50 & 19 & 10 & 13 & \(\bullet\) & \(\bullet\) & 4 \\
57 & 51 & 39 & 47 & 50 & \(\bullet\) & 45 & \(\bullet\) \\
67 & 28 & 51 & 52 & 18 & 52 & \(\bullet\) & \(\bullet\) \\
\end{longtable}}

\subsection*{The worked example of the certificate}

Here one cell is unoccupied, so the closure is built on the set of occupied cells and the zero-companion rule of Section~\ref{sec:prelim} participates: a rectangle one of whose corners is the hole has the diagonal through the hole of value \(0\), and its opposite diagonal is certified at its own prescribed value, identified when it is a selected two-edge and declared orthogonal when it is not. Of the \(120\) opposite diagonals none is a selected two-edge, so only the second case occurs. The replay is the third line of Table~\ref{tab:certs}: of the \(52\) two-edges \(11\) are degenerate (namely \(8\) in a row and \(3\) in a column) and resolved by the line rule, \(6\) occur as the two diagonals of the \(3\) complementary rectangles, and the remaining \(35\), all nondegenerate, are resolved by rectangle transfer; every selected two-edge is thereby resolved, and the \(94\) selected edges determine \(94\) distinct equivalence classes, none of them the hole. The hole lies in the row indexed by the edge \(\{2,7\}\) and in the column of the vertex \(4\).

This is the configuration for which the certificate needs the zero-companion rule, and it is the worked example of the grounded-component normal form of \S\ref{subsec:grounded}. The quotient graph \(\Gamma\) of \eqref{eq:X} has \(\binom{94}2=4371\) variables and \(322\) components. Without the hole grounds \(16\) of its components remain ungrounded -- twelve of size \(1\), three of size \(2\) and one of size \(3\), that is \(21\) variables -- and the closure of the other four rules stops exactly there, at \(4350\) of the \(4371\) pairs. The \(16\) pairs that a rectangle with the hole as a corner grounds directly are the pairs consisting of a class with a cell in the row of the hole and a class with a cell in its column, and the remaining \(5\) follow from them by saturation and rectangle transfer; those grounds reach all \(16\) components, so that with the fifth rule the closure reaches its least fixed point at the \(\binom{94}2=4371\) certified pairs of Table~\ref{tab:certs}. The checker \texttt{ancillary/certificate\_checker.py} treats the hole in the corrected form of Section~\ref{sec:prelim}, reports the \(120\) rectangles the rule addresses and no uncertified pair, and reproduces every number above.

\section{The explicit \(78\times13\) configuration}\label{app:7813}

The configuration is recorded cell by cell in the data sheet \sheet{7813} \cite{sheets}; it has the edges of \(K_{13}\) as rows and its vertices as columns, the incidence graph of \(K_{13}\) as one-edge set with \(|E_1|=156=z(78,13)\), \(429\) two-edges with pairwise disjoint supports, no hole, and rank \(R=585=Z(13)\), as listed in Table~\ref{tab:shape}, and its replay data are listed in Table~\ref{tab:certs}. It is obtained from the certified \(n=14\) construction of Section~\ref{sec:even} at \(q=7\) by deleting the star of the vertex \(\infty\), that is the \(13\) rows indexed by its edges together with the corresponding column: of the \(546\) two-edges of that construction, \(312\) lose no half and \(234\) lose exactly one half, while none loses both halves. The restricted grid inherits the \(312\) two-edges and its \(234\) orphaned cells are re-paired into \(117\) further two-edges by a randomized local search scored by the \((\mathrm{RW}3^+)\) closure; the labels are renumbered \(1,\dots,429\) in order of first appearance, and the counts are collected in Table~\ref{tab:prov}.

\section{The explicit \(66\times12\) configuration}\label{app:6612}

The configuration is recorded in the data sheet \sheet{6612} \cite{sheets}; it has \(|E_1|=132=z(66,12)\), \(330\) disjoint two-edges, no hole and rank \(R=462=Z(12)\), as listed in Tables~\ref{tab:shape} and~\ref{tab:certs}. It is obtained from the \(78\times13\) configuration of the data sheet \sheet{7813} by iterating once more the same operation, deleting the star of a vertex of \(K_{13}\) together with its column: of the \(429\) two-edges of that configuration, \(238\) lose no half, \(184\) lose exactly one half and \(7\) lose both halves, so the restricted grid inherits \(238\) two-edges and its \(184\) orphaned cells are re-paired into \(92\) further two-edges (Table~\ref{tab:prov}). This order is treated after \(n=13\) because of that dependence; the operation uses no perfect one-factorization certificate at all here, its only input being the certified odd-prime case \(q=7\). The same value is also obtained in one step, by deleting the stars of two vertices of the \(n=14\) construction at once and repairing the \(324\) orphaned cells; both witnesses pass the same replay.

\section{The explicit \(120\times16\) configuration}\label{app:12016}

The configuration is recorded in the data sheet \sheet{12016} \cite{sheets}; it has \(|E_1|=240=z(120,16)\), \(840\) disjoint two-edges, no hole and rank \(R=1080=Z(16)\), as listed in Tables~\ref{tab:shape} and~\ref{tab:certs}. It is obtained from the nested construction at \(q=9\), of displayed length \(R(G)=1530\), by deleting the stars of two vertices of \(K_{18}\) (the \(33\) rows indexed by the edges through them and the two corresponding columns): of the \(1224\) two-edges of that configuration, \(528\) lose no half, \(624\) lose exactly one half and \(72\) lose both halves, so the restricted grid inherits \(528\) two-edges and its \(624\) orphaned cells are re-paired into \(312\) further two-edges (Table~\ref{tab:prov}). The parent is the nested witness of order \(18\), which is \emph{not} certified by \((\mathrm{RW}3^+)\): its closure leaves \(7344\) of the \(\binom{1530}2=1\,169\,685\) pairs of selected edges uncertified. It is also not the \(153\times18\) array of Appendix~\ref{app:15318}, which is a different configuration of the same shape, obtained there from the certified \(q=11\) construction; here the operation therefore repairs the certificate as well as the data.

\section{The explicit \(136\times17\) configuration}\label{app:13617}

The configuration is recorded in the data sheet \sheet{13617} \cite{sheets}; it has \(|E_1|=272=z(136,17)\), \(1020\) disjoint two-edges, no hole and rank \(R=1292=Z(17)\), as listed in Tables~\ref{tab:shape} and~\ref{tab:certs}. It is obtained from the same nested construction at \(q=9\) as Appendix~\ref{app:12016}, again not the array of Appendix~\ref{app:15318}, by deleting the star of a vertex of \(K_{18}\): of the \(1224\) two-edges of that configuration, \(816\) lose no half and \(408\) lose exactly one half, none losing both halves, so the restricted grid inherits the \(816\) two-edges and its \(408\) orphaned cells are re-paired into \(204\) further two-edges (Table~\ref{tab:prov}); the parent here is not certified either.

\section{The explicit \(153\times18\) configuration}\label{app:15318}

The configuration is recorded in the data sheet \sheet{15318} \cite{sheets}; it has \(|E_1|=306=z(153,18)\), \(1224\) disjoint two-edges, no hole and rank \(R=1530=Z(18)\), as listed in Tables~\ref{tab:shape} and~\ref{tab:certs}. It is obtained from the certified \(231\times22\) configuration of Section~\ref{sec:even} at \(q=11\), that is the case \(n=22\) of Theorem~\ref{thm:unified}, by deleting the stars of four vertices of \(K_{22}\) (the \(78\) rows indexed by the edges through them and the four corresponding columns): of the \(2310\) two-edges of that configuration, \(560\) keep both halves, \(1328\) lose exactly one half and \(422\) lose both halves. The restricted grid retains \(557\) of the \(560\); its \(1328\) orphaned cells are re-paired into \(664\) further two-edges and the \(6\) cells of the \(3\) surviving two-edges that are broken are re-paired into \(3\) more, giving \(557+664+3=1224\) two-edges (Table~\ref{tab:prov}). No hypothesis on \(K_{18}\) itself is used: the only certified input is the odd-prime case \(q=11\).

\section{The explicit \(190\times20\) configuration}\label{app:18920}

The configuration is recorded in the data sheet \sheet{18920} \cite{sheets}; it has \(|E_1|=380=z(190,20)\), \(1710\) disjoint two-edges, no hole and rank \(R=2090=Z(20)\), as listed in Tables~\ref{tab:shape} and~\ref{tab:certs}. It is obtained from the certified \(231\times22\) configuration at \(q=11\) by deleting the stars of two vertices of \(K_{22}\) (the \(41\) rows indexed by the edges through them and the two corresponding columns): of the \(2310\) two-edges of that configuration, \(1200\) keep both halves, \(1020\) lose exactly one half and \(90\) lose both halves. The restricted grid retains \(1196\) of the \(1200\); its \(1020\) orphaned cells are re-paired into \(510\) further two-edges and the \(8\) cells of the \(4\) surviving two-edges that are broken are re-paired into \(4\) more, giving \(1196+510+4=1710\) two-edges (Table~\ref{tab:prov}).

\section{The explicit \(210\times21\) configuration}\label{app:21021}

The configuration is recorded in the data sheet \sheet{21021} \cite{sheets}; it has \(|E_1|=420=z(210,21)\), \(1995\) disjoint two-edges, no hole and rank \(R=2415=Z(21)\), as listed in Tables~\ref{tab:shape} and~\ref{tab:certs}. It is obtained from the certified \(231\times22\) configuration at \(q=11\) by deleting the star of a vertex of \(K_{22}\) (the \(21\) rows indexed by its edges and the corresponding column): of the \(2310\) two-edges of that configuration, \(1680\) keep both halves and \(630\) lose exactly one half, none losing both, so the restricted grid inherits the \(1680\) two-edges and its \(630\) orphaned cells are re-paired into \(315\) further two-edges, giving \(1680+315=1995\) (Table~\ref{tab:prov}). No inherited two-edge is modified in this instance.

\end{document}